\documentclass{amsart}
\usepackage{graphicx}

\makeatletter %较粗的内积符号
\newcommand*\bigcdot{\mathpalette\bigcdot@{.5}}
\newcommand*\bigcdot@[2]{\mathbin{\vcenter{\hbox{\scalebox{#2}{$\m@th#1\bullet$}}}}}
\makeatother

\usepackage{amsmath}
\usepackage{amsthm}
\usepackage{amsfonts}
\usepackage{mathrsfs}
\usepackage{xcolor}
\usepackage{amsmath}
\usepackage{amsfonts}

\usepackage{enumerate}
\usepackage{amssymb}
\usepackage{amsthm}%必须放在amsmath之后
\usepackage{bm}

\usepackage[colorlinks=true,linkcolor=blue]{hyperref}%这个放在最后防止冲突

\usepackage{latexsym, tikz-cd}
\usepackage[normalem]{ulem}%提供各种各样的标记线(可以用于数学公式环境中)

\newtheorem{theorem}{Theorem}[section]
\newtheorem{proposition}[theorem]{Proposition}
\newtheorem{corollary}[theorem]{Corollary}
\newtheorem{lemma}[theorem]{Lemma}
\newtheorem{conjecture}[theorem]{Conjecture}
\newtheorem*{conjectureprime}{Conjecture \ref{js14}$'$}

\newtheorem*{openquestion}{Open Question}
\theoremstyle{definition}
\newtheorem{definition}[theorem]{Definition}
\theoremstyle{remark}
\newtheorem{remark}[theorem]{Remark}

\newcommand{\nwc}{\newcommand}

\nwc{\Oph}{\operatorname{Op}_\hbar}

\renewcommand{\Re}{\operatorname{Re}}

\DeclareMathOperator{\Tr}{Tr}

\newcommand{\La}{\Lambda}

\renewcommand{\phi}{\varphi}

\nwc{\gl}{\langle}
\nwc{\gr}{\rangle}

\newcommand{\R}{{\mathbb R}}

\newcommand{\C}{{\mathbb C}}

\newcommand{\Z}{{\mathbb Z}}

\newcommand{\D}{{\mathbb D}}
\nwc{\rest}{\restriction}

\nwc{\defeq}{\stackrel{\rm{def}}{=}}
\renewcommand{\d}{\partial}

        \definecolor{pink}{rgb}{1,0,1}

\makeatletter

\@addtoreset{equation}{section}
\makeatother
\title[Steklov Spectral Geometry for Annular Surfaces]{Steklov Spectral Geometry for Annular Surfaces: Inverse spectral results and isospectral compactness}
\date{}

\author{Yujun Jin, Zuoqin Wang}
\thanks{Partially supported by   NSFC no. 12571064.}
\address{School of Mathematical Sciences\\
	University of Science and Technology of China\\
	Hefei, 230026\\ P.R. China\\}
\email{byjyj@mail.ustc.edu.cn}

\address{School of Mathematical Sciences\\
	University of Science and Technology of China\\
	Hefei, 230026\\ P.R. China\\}
\email{wangzuoq@ustc.edu.cn}

\begin{document}

\maketitle

\begin{abstract}
We study the inverse spectral problems for the Steklov spectrum on compact surfaces with boundary. We prove that among flat annular surfaces, the lateral surface of a conical frustum is uniquely determined by its Steklov spectrum. As a consequence, each circular annulus is uniquely determined among all planar domains, providing the first example of a non-simply connected Euclidean domain with this property. Furthermore, we show that any family of Steklov isospectral flat annular surfaces is compact in the $C^\infty$  topology, extending previous results for simply connected planar domains. These results are established through  a detailed analysis of the spectral zeta function and the zeta-regularized determinant associated with the Dirichlet-to-Neumann operator for annular surfaces.
\end{abstract}

\section{Introduction}
Let $(\Omega,g)$ be a compact surface with smooth boundary $M=\d \Omega$, which is topologically a union of disjoint circles. Consider the Steklov eigenvalue problem on $\Omega$, 
\begin{equation}
\left\{\begin{array}{ll}
    \Delta_g u=0, & \text{in }\Omega, \\
    \frac{\d u}{\d n}=\sigma u, & \text{on } M.
\end{array}
\right.
\end{equation}
Here, $n$ denotes the outward unit normal vector. This problem admits a discrete spectrum
\begin{equation}\label{eigenvalues}
0=\sigma_0\leq\sigma_1\leq\sigma_2\leq\dots,
\end{equation}
which is non-negative and diverges to infinity.   These Steklov eigenvalues also appear as the eigenvalues of the \emph{Dirichlet-to-Neumann} (DtN) operator $\Lambda_{M,g}: C^{\infty}(M)\longrightarrow C^\infty(M)$, defined by
\[
\Lambda_{M,g} f:=\frac{\d (\mathcal{H}f)}{\d n}\Big|_{M},
\]
where $\mathcal{H}f \in C^\infty(\Omega)$ denotes the harmonic extension of $f\in C^\infty(M)$ to $\Omega$. It is well known that $\Lambda_{M,g}$ is an elliptic pseudodifferential operator of order one on $M$ \cite{LU1989}. 

The inverse Steklov spectral problem concerns the extent to which the geometry of $\Omega$ can be recovered from the spectrum \eqref{eigenvalues}. Several related questions are formulated in \cite{JollivetSharafutdinov2014, GirouardPolterovich2017, Colbois2024}. A fundamental question is whether the Steklov spectrum uniquely determines the underlying surface up to isometry.
In general, the answer is negative. Indeed, for any conformal metric
\[
g'=\rho g,\quad \rho\in C^\infty(\Omega;\R_+)
\]
such that $\rho|_M = 1$, it is straightforward to verify that $\Lambda_{M,g'} = \Lambda_{M,g}$.
As a result, one can easily construct infinitely many Steklov isospectral metrics on $\Omega$ that are pairwise non-isometric. 

Motivated by this observation, Jollivet and Sharafutdinov formulated the
corresponding inverse problem in the simply connected case
\cite{JollivetSharafutdinov2014}. A broader version for compact Riemannian
surfaces was subsequently stated in
\cite{GIROUARD_PARNOVSKI_POLTEROVICH_SHER_2014}. We record it in the following
form.
\begin{conjecture}\label{js14}
 Two Riemannian metrics $g_1$ and $g_2$ on a given smooth surface $\Omega$ have the same Steklov spectrum if and only if they are $\sigma$-isometric, i.e., there exist a diffeomorphism $\Phi:\Omega\to\Omega$ and a function $\rho\in C^\infty(\Omega;\R_+)$ such that $\rho|_M=1$ and
 \[
 g_2=\rho\Phi^*g_1.
 \]
\end{conjecture}

The first goal of this paper is to investigate this conjecture for flat surfaces, and in particular, for Euclidean planar domains. Note that if two conformally equivalent flat metrics $g$ and $\rho g$  satisfy   $\rho|_M=1$, it necessarily follows that $\rho=1$ throughout  $\Omega$. Therefore, in this setting the conjecture reduces to:
\begin{conjectureprime}[Conjecture 1.1, flat version]\label{conjecture2}
    Two flat surfaces are Steklov isospectral if and only if they are isometric.
\end{conjectureprime}

This problem is analogous to Kac's famous inverse spectral question \cite{kac}, ``Can one hear the shape of a drum?'', which asks whether planar domains are uniquely determined by their Laplacian spectra up to isometry. Although answered negatively by Gordon, Webb, and Wolpert \cite{GWW92}, Kac's problem remains open for planar domains with smooth boundary. As for the Steklov problem, the only smooth Euclidean domains currently known to be uniquely determined by their Steklov spectra are planar disks, due to Edward \cite{ED93}, and three-dimensional balls, due to Polterovich and Sher \cite{Polterovich2015}. See also Wang and Yun \cite{Wang-Yun} for analogous   results in 3-dimensional simply connected space forms. 
On the other hand, positive results have been obtained when the inverse 
Steklov problem is restricted to several special classes. Examples include
warped product manifolds \cite{DaudeKamranNicoleau2021, Gendron2020} and
certain special planar polygons \cite{DrydenGordonMorenoRowlettVillegasBlas2026}.

In this work, we focus on a specific class of surfaces.
\begin{definition}
	We call a compact, orientable, connected  Riemannian surface $(\Omega, g)$ an \emph{annular surface} if it is  of genus zero with exactly two smooth boundary components. 
We denote by $\mathcal{A}$ the set of all smooth \emph{annular surfaces}.
\end{definition}
According to the Osgood--Phillips--Sarnak uniformization theorem (see Section \ref{uniformazation}), any annular surface is conformal to some flat cylinder $\mathcal C_l$  with geodesic boundary.   

A typical annular surface is  the lateral surface of a conical frustum $A_{r_1,r_2,h}$ (see Section \ref{application} for details), endowed with the standard Euclidean metric. 
Our first main result in this paper is to prove that these special annular surfaces are spectrally determined among flat annular surfaces, thereby providing evidence for Conjecture \ref{js14}$'$.
\begin{theorem}\label{main1}
  Among all flat surfaces in $\mathcal A$, the lateral surface of each conical
    frustum $A_{r_1,r_2,h}$ is uniquely determined by its Steklov spectrum.
\end{theorem}

By \cite{GIROUARD_PARNOVSKI_POLTEROVICH_SHER_2014}, for any compact
smooth surface $\Omega$, the Steklov spectrum determines the number of connected
components of its boundary $M$. Thus Theorem \ref{main1} directly yields the following result.
\begin{corollary}
    	Each annulus bounded by two concentric circles is uniquely determined by its Steklov spectrum among all planar domains with smooth boundary.
\end{corollary}
To our knowledge, this provides the first example of a non-simply connected Euclidean domain that is uniquely determined by its Steklov spectrum. 

The main tools in the proof of   Theorem \ref{main1} are the spectral zeta function  $\zeta_{\Sigma,g}(s)$  and the zeta-regularized determinant $\mathrm{det}'\Lambda$ associated with the DtN operator $\Lambda_{\Sigma,g}$ (see Section \ref{defzetainvariant} for definitions). Our first observation is that, according to Guillarmou--Guillopé's conformal invariance formula for the determinant \cite{8180465}, the Steklov spectrum determines the conformal class of an annular surface (see Corollary \ref{specdetcon}).  More precisely, once the spectrum is fixed, all such surfaces can be identified with  a fixed cylinder $\mathcal C_l$ endowed with metrics $g$ conformal to the standard product metric $g_0$. This is done by explicitly computing the log-determinant of such a cylinder.    We then  derive explicit formulas for the zeta values $\zeta_{\Sigma,g}(-2m)$, $m \in \mathbb{Z}_+$, in terms of the Fourier coefficients of the conformal factors. This extends earlier computations in \cite{Malkovich2015} for topological disks to annular surfaces.
The last step in the proof is to show that among all flat surfaces in $\mathcal{A}$, the surface $A_{r_1,r_2,h}$ uniquely minimizes $\zeta_{\Sigma,g}(-2)$.

In the remainder of this paper, we address a related weaker question: the compactness of the family of Steklov isospectral planar domains in the $C^\infty$ topology. We shall use the Cheeger--Gromov
notion of precompactness, specialized   to compact
surfaces and with   specified regularity of convergence. More precisely, let $\mathcal{F}$ be a family of compact surfaces with boundary. We say that $\mathcal F$ is compact in the $C^\infty$ topology if, for every
sequence $S_k$ in $\mathcal F$, there exists a subsequence $S_{k_j}$, a compact reference surface
$\Omega$ together with smooth metrics $g_{k_j}$ on $\Omega$, and a family of 
isometries
\[
P_{k_j}:(\Omega,g_{k_j})\longrightarrow S_{k_j},
\]
such that $g_{k_j}$ converges to a smooth metric $g$ in $C^\infty$. For
$s\geq 0$, compactness in the $H^s$ topology is defined in the same way,
with the convergence $g_{k_j}\to g$ understood in $H^s(\Omega)$ with
respect to any fixed smooth background metric.

While the compactness problem for Laplace-isospectral metrics has been answered positively on surfaces \cite{OPSmodulispace,OSGOOD1988212,Kim2008}, only limited results are known in higher dimensions (see, e.g., \cite{ChangYang,ChenXu1996,LiuWang2019}). By contrast, the analogous question in the Steklov setting remains largely unexplored. 
One of the difficulties lies in the lack of sufficiently strong spectral invariants. The classical approach to establishing the compactness of a family of isospectral metrics relies on controlling Sobolev norms of  geometric quantities via the coefficients of the heat trace expansion, which are spectral invariants. While this method is effective for elliptic differential operators such as the Laplacian, it encounters serious obstacles in the case of the DtN operator. Indeed, as a pseudodifferential operator, only some of the heat coefficients of the DtN operator are local \cite{gilkey04}. Here, ``local" means that these coefficients are given by sums of integrals of quantities that depend only on the metric near the boundary. Consequently, these coefficients alone do not currently provide the necessary control of the relevant Sobolev norms.

Nevertheless, for simply connected planar domains, this difficulty can be circumvented by exploiting the spectral zeta function at negative integers. Edward \cite{Edward01011993} showed that any family of Steklov isospectral simply connected planar domains is compact in $H^{\frac52-\epsilon}$ for any $\epsilon>0$. A crucial step in the proof is the use of the values $\zeta_{\partial \D,g}(s)$ for $s=-1,-2,-3, -4$, where $ \D$ denotes the unit disk. This result was later strengthened to $C^\infty$-compactness by Jollivet and  Sharafutdinov by considering all values $\zeta_{\d \D,g}(-2m)$, $m\in\Z_+$ \cite{JOLLIVET20181712}. More precisely, they prove that, for all $m\in \Z_+$, there exist positive constants $c_m$ such that 
\begin{equation}\label{jsestimate}
   c_m \sum_{k=m+1}^\infty k^{2m+1}|(\widehat{b^m})_k|^2\leq \zeta_{\d\D,g}(-2m),
\end{equation}
where $b$ is a positive smooth function defined by $g=b^{-2}g_E$, with $g_E$ the canonical Euclidean metric on $\D$, and $(\widehat{b^m})_k$ denotes the $k$-th Fourier coefficient of $b^m|_{\d \D}$.  Estimates \eqref{jsestimate} give uniform control of the
higher Fourier coefficients of the boundary conformal factor. 
It remains only to handle the first Fourier coefficient. This can be achieved by replacing \(g\) with a conformal metric \(g'\) such that \((\mathbb{D},g')\) is isometric to \((\mathbb{D},g)\) and its first Fourier coefficient vanishes. The existence of such a normalization is due to Edward
\cite{Edward01011993}.

For the case of multiply connected planar domains, the following open question was posed by Colbois, Girouard, Gordon, and Sher:
\begin{openquestion}[\cite{Colbois2024}, Open Question 8.5]
	Are families of multiply connected, compact, Steklov isospectral planar domains necessarily compact in the $C^\infty$ topology?
\end{openquestion}

Our second main result in this paper is   the following compactness theorem, which in particular provides an affirmative answer to the above open question  for doubly connected planar domains. 
\begin{theorem}\label{main::compactness}
	Any family of Steklov isospectral flat surfaces in $\mathcal{A}$ is compact in the $C^\infty$ topology. 
\end{theorem}

The first step in our proof is to extend the estimate \eqref{jsestimate}  to  annular surfaces. More precisely, for a cylinder with boundary $\Sigma$, we show that, for all $m\in \Z_+$, there exist positive constants $d_{m,l}$ and $d'_{m}$ such that
\[
\|{a}^m\|_{H^{m+{\frac12}}(\Sigma)}^2\leq d_{m,l}(\zeta_{\Sigma,g}(-2))^{m}+d'_{m}\zeta_{\Sigma,g}(-2m),
\]
where $a=e^{-\phi}|_{\Sigma}$ is the positive smooth boundary function
associated with the conformal metric $g=e^{2\phi}g_0$, 
and the constants $d_{m,l}$ and $d'_{m}$ depend only on $m$ and   the Steklov spectrum. We then use the value $\zeta_{\Sigma,g}(-1)$ to give a positive lower bound for $a$. Together with the values $\zeta_{\Sigma,g}(-2m)$, we obtain bounds for all Sobolev norms of $a$, and the compactness follows.

For surfaces with  at least three boundary components,   the compactness analysis of Steklov isospectral families becomes more complicated, since the Steklov spectrum may no longer determine the conformal class. In the case of hyperbolic surfaces,  
Wentworth \cite{WentworthDN} proved that on  a fixed genus-zero topological surface with at least three boundary components, any Steklov isospectral family of hyperbolic metrics with geodesic boundary is compact in the $C^\infty$ topology. In a subsequent paper, we will address the Steklov isospectral compactness problem for flat surfaces with at least three  boundary components. 

The paper is organized as follows. In Section \ref{sec::compactsurfaces}, we first prove the main trace identities for zeta values of the DtN operator (Proposition \ref{traceidentities}). We also establish a Polyakov-type formula (Proposition \ref{prop::zero-mode-correction}) for general conformally covariant pseudodifferential operators, using which we provide an alternative proof of Guillarmou--Guillopé's conformal invariance formula for the determinant. In Section \ref{Section::OPS}, we recall the Osgood--Phillips--Sarnak uniformization for annular surfaces, and then use cylindrical uniformization to show that the Steklov spectrum determines the conformal class. We also obtain a compactness-in-conformal-class result (Corollary \ref{compactinconfclass}) in the spirit of \cite{ChangYang}, which may be of independent interest.  In Section \ref{sec::inv}, we use the trace identities from Section \ref{sec::compactsurfaces} to compute the zeta invariants in the cylindrical model and  prove Theorem \ref{main1}. Finally, in Section \ref{compactness}, we establish the desired uniform estimates for the conformal factors of Steklov isospectral annular surfaces and complete the proof of Theorem \ref{main::compactness}.

\section{Spectral invariants of the DtN operator on compact surfaces}\label{sec::compactsurfaces}

\subsection{The full symbols}
Let $(\Omega,g)$ be a compact surface with smooth boundary $M$. 
Denote the connected components of $M$ by  $M_i$ ($i=1,\dots,n$), and denote the length of $M_i$ by $L_g(M_i)$. Let $\Lambda_{M,g}$ be the DtN operator on $M$ and let $D_{M,g}=\frac{1}{\sqrt{-1}}\d_\tau$ be the tangential derivative, where $\tau$ is the arc-length parameter of the boundary $M$ with respect to the metric $g$. If the underlying manifold is clear, we simply write them as $\Lambda_g$ and $D_g$. 

The following lemma will be used in the subsequent trace computations as a technical tool to prove that certain pseudodifferential operators are of trace class. A related computation in the simply connected case is carried out in \cite[Page 283]{Sharafutdinov18}. The result below extends that analysis to general compact surfaces with boundary. 

\begin{lemma}\label{symbols}
  For any smooth function $a$ on $M$, the operators $(aD_g)^2$ and
  $(a\Lambda_g)^2$ have the same full symbol.
\end{lemma}

\begin{proof}
The assertion is local on the boundary. Fix a boundary component $M_i$, which is topologically a circle. Let $\D\subset \R^2$ be a unit disk. Choose a diffeomorphism from a collar neighborhood of $\partial\D$ onto a
collar neighborhood of $M_i$ such that, in this collar chart, the pull-back of the metric $g$
is written as 
$e^{2\phi_i}g_E$
near $\partial\D$, where $g_E$ is the Euclidean metric on $\D$. Extend
$\phi_i$ smoothly to $\D$ and set
\[
g_i=e^{2\phi_i}g_E .
\]
By the locality of the full symbol of the Dirichlet-to-Neumann operator
\cite{LU1989}, the full symbol of $\Lambda_{M,g}$ along $M_i$ agrees, under this
collar identification, with that of $\Lambda_{\d \D ,g_i}$ along $\partial\D$. The
same holds for the tangential operators $D_{M,g}$ and $D_{\d \D,g_i}$.

Let $b_i:=e^{-\phi_i}|_{\partial\mathbb D} $. We have
\[
\Lambda_{\d \D, g_i}=b_i\Lambda_{\partial\mathbb D,g_E},
\qquad
D_{\d \D, g_i}=b_iD_{\partial\mathbb D,g_E}.
\]
Thus, after identifying $a|_{M_i}$ with its pullback to $\partial\D$ and
setting
\[
c:=(a|_{M_i})b_i,
\]
it suffices to compare the full symbols of
\[
(c\Lambda_{\partial\D,g_E})^2
\quad\text{and}\quad
(cD_{\partial\D,g_E})^2 .
\]

By \cite{ED93},
\[
\sigma_{\mathrm{full}}(\Lambda_{\partial\mathbb D,g_E})=|\xi|,
\qquad
\sigma_{\mathrm{full}}(D_{\partial\mathbb D,g_E})=\xi .
\]
Using the composition formula for pseudodifferential operators $P,Q$,
\[
\sigma_{\mathrm{full}}(P\circ Q)(x,\xi)\sim\sum_{\omega}\d_{\xi}^\omega(\sigma_{\mathrm{full}}(P)) D_x^\omega(\sigma_{\mathrm{full}}(Q)),
\]
we get
\[
\sigma_{\mathrm{full}}\bigl((c\Lambda_{\partial\mathbb D,g_E})^2\bigr)
=c^2\xi^2+c(Dc)\xi
=\sigma_{\mathrm{full}}\bigl((cD_{\partial\mathbb D,g_E})^2\bigr).
\]
Hence the two operators have the same full symbol along $M_i$. Repeating the
argument on each boundary component gives the result.
\end{proof}

Given any positive elliptic pseudodifferential operator $A\in\Psi^d(M)$  with $d>0$, we will denote the  $j$-th eigenvalue of $A$ by $\sigma_j(A,M)$. 
As a consequence of Lemma \ref{symbols},  the eigenvalues of both $\Lambda_g$ and $|D_g|$ admit the following asymptotic formula
\cite{GIROUARD_PARNOVSKI_POLTEROVICH_SHER_2014}, 
\begin{equation}\label{eq::asymptoticformula}
	\sigma_j(\La_g,M)=\sigma_j(|D_g|,M)+O(j^{-\infty}).
\end{equation}
It is well known that the total length of the boundary 
\[L_g(M)=L_g(M_1)+L_g(M_2)+\cdots +L_g(M_{n})\] 
can be obtained from the Steklov spectrum of $\Omega$. Indeed, the coefficients in the asymptotic expansion (as $t$ tends to zero) of the heat trace  $\Tr e^{-\Lambda_gt}$ give $L_g(M)$. Moreover, by using \eqref{eq::asymptoticformula} and Dirichlet’s theorem on simultaneous approximation, Girouard, Parnovski, Polterovich and Sher proved that the number of connected components of $M$ and the length of each component $L_g(M_i)$ are also Steklov spectral invariants \cite{GIROUARD_PARNOVSKI_POLTEROVICH_SHER_2014}. 

\subsection{Determinants and zeta invariants of $\Lambda_g$}\label{defzetainvariant}
For any positive elliptic pseudodifferential operator $A\in \Psi^d(M)$,   its complex powers can be defined via the spectral theorem as
\begin{equation}\label{complexpowerAs}
A^s:=\int_{\sigma(A)\setminus \{0\}} \lambda^sdE(\lambda),
\end{equation}
where $E(\lambda)$ denotes the projection-valued  spectral family associated with $A$, and $\sigma(A)\subset \mathbb{R}_{\geq 0}$ is the spectrum of $A$. This definition is consistent with the classical construction of complex powers for more general pseudodifferential operators introduced by Seeley \cite{Seeley1967}, which is based on contour integrals of the resolvent. For non-invertible $A$, the definition formula  \eqref{complexpowerAs} allows $0$ to remain in the spectrum and interprets $A^s$ as acting trivially on $\ker A$.

Now consider the \emph{spectral zeta function} associated with the DtN operator $\Lambda_g$, which is defined as  
\begin{equation}
    \zeta_{M,g}(s):=\Tr \Lambda_g^{-s}=\sum_{\sigma_j(\La_g,M)>0} \big(\sigma_j(\La_g,M)\big)^{-s}.
\end{equation}
We will simply use the notation $\zeta_M$ when there is no ambiguity. 
It is well-known that $\zeta_{M,g}(s)$ is absolutely convergent and therefore analytic in the region $\{\Re s>1\}$. It can also be extended to a meromorphic function on $\C$ with a unique simple pole at $s=1$. In particular, $\zeta_{M,g}(s)$ is analytic at $s=0$, and the \emph{(zeta-regularized) determinant} of $\La_g$ is defined by
\begin{equation}\label{def::det}
    {\det}'\La_g:=\exp (-\zeta_{M,g}'(0)).
\end{equation}

If  we fix a reference surface $(\Omega,g)$ with boundary $M$, each surface conformal to $(\Omega, g)$ is isometric to $(\Omega,h)$ for some metric $h=e^{2\phi}g$, $\phi\in C^\infty(\Omega;\R)$. Setting $a=e^{-\phi}|_{M}>0$, we have
\[
\Lambda_h=a\Lambda_g,\quad D_h=aD_g.
\]

In the simply connected case, the values of the spectral zeta function at negative even integers are referred to in \cite{Malkovich2015} as \emph{zeta invariants} of the Steklov spectrum. The result below extends the corresponding formulas to general compact surfaces with boundary.
\begin{proposition}\label{traceidentities}
   With the notation above, after meromorphic continuation, the following identity holds for all
$s\in \mathbb C$:
\begin{align*}
\zeta_{M,h}(s)-2\zeta_{R}(s)\sum_{i=1}^n\left(\frac{2\pi}{L_h(M_i)}\right)^{-s}&=\Tr (\La_h^{-s}-|D_h|^{-s})\\
&=\Tr \left(\left(a^{\frac12}\La_ga^{\frac12}\right)^{-s}-\left|a^{\frac12}D_{g}a^{\frac12}\right|^{-s}\right),
\end{align*}
where $\zeta_R(s)$ denotes the Riemann zeta function. 
In particular, for each $m\in\Z_+$, one has
\begin{equation}\label{generalzeta-2m}
	\zeta_{M,h}(-2m)=\Tr\left((a\La_g)^{2m}-(aD_g)^{2m} \right),
\end{equation} and 
\begin{equation}\label{generalzeta-1}
    \zeta_{M,h}(-1)=-\frac{\pi}{3}\sum_{i=1}^n \frac1{L_h(M_i)}+\Tr \left(a^{\frac12}\La_ga^{\frac12}-|a^{\frac12}D_{g}a^{\frac12}|\right).
\end{equation} 
\end{proposition}
\begin{proof}
By  Lemma \ref{symbols},  $\sigma_{\mathrm{full}}(\Lambda_h^2)=\sigma_{\mathrm{full}}(D_h^2)$. So $\sigma_{\mathrm{full}}(\Lambda_h)=\sigma_{\mathrm{full}}(|D_h|)$. As a result, for any $s\in\C$, the operator \[\La_h^{-s}-|D_h|^{-s}\in \Psi^{-\infty}(M)\] is a smoothing operator and thus a trace class operator. For $\operatorname{Re}s$ sufficiently large, the definition of spectral zeta function gives
\begin{equation*}
    \zeta_{M,h}(s)-\zeta_{|D_h|}(s)=\Tr (\La_h^{-s}-|D_h|^{-s}).
\end{equation*}
Since the right-hand side is an entire function of $s$, the same identity holds for all
$s\in\mathbb C$ after analytic continuation.

A direct computation shows that  the eigenvalues of the operator $D_h$ are precisely the numbers 
\[
\frac{2\pi k}{L_h(M_i)}, \qquad k \in \mathbb Z, i=1,\dots,n
\] 
with the corresponding eigenfunctions (see \cite[Lemma 2.1]{Sharafutdinov18})
\begin{equation*}
        {E}_{i,k}(\tau)=\begin{cases}
            \exp{\left(\frac{2\pi k\sqrt{-1}}{L_{h}(M_i)}\int_0^{\tau}a^{-1}ds_g\right)}, & \text{on }M_i,\\
            0,& \text{else.}
        \end{cases}
\end{equation*}
%Moreover, these functions $\{{E}_{i,k}\}$ form an orthogonal basis for $L^2(M, h)$ according to \cite[Lemma 2.1]{Sharafutdinov18}. 
Therefore, the spectral zeta function associated with the operator $|D_h|$ is given by
\begin{equation*}
    \zeta_{|D_h|}(s)=2\zeta_{R}(s)\sum_{i=1}^{n}\left(\frac{2\pi}{L_h(M_i)}\right)^{-s}
\end{equation*}
for each $s\in \C$. This proves the first equality. 

Since multiplication by $a^{\frac12}$ is a unitary operator from $L^2(M,g)$ to $L^2(M,h)$, it follows that, whenever an operator $A_h$ is self-adjoint
on $L^2(M,h)$, the conjugated operator 
$A_g:=a^{-\frac12}A_ha^{\frac12}$ is self-adjoint on $L^2(M,g)$. Moreover, their spectral families satisfy
\[
E_{A_g}(\lambda)=a^{-\frac12}E_{A_h}(\lambda)a^{\frac12}.
\]
Therefore, by the spectral theorem, for any Borel function $f$ for which
the functional calculus is defined,
\[
f(A_g)=a^{-\frac12}f(A_h)a^{\frac12}.
\]
Applying this to $\Lambda_h$ and $|D_h|$, we obtain
\[
a^{-\frac12}\La_h^{-s}a^{\frac12}=\left(a^{\frac12}\La_ga^{\frac12}\right)^{-s}\quad\text{and}\quad a^{-\frac12}|D_h|^{-s}a^{\frac12}=\left|a^{\frac12}D_{g}a^{\frac12}\right|^{-s}.
\]
Hence
\[
\Tr (\La_h^{-s}-|D_h|^{-s})
=\Tr \left(\left(a^{\frac12}\La_ga^{\frac12}\right)^{-s}-\left|a^{\frac12}D_{g}a^{\frac12}\right|^{-s}\right).
\]

In particular, at the trivial zeros $s = -2m$ ($m \in \Z_+$) of the Riemann zeta function $\zeta_R$, we have 
\begin{equation*}
	\zeta_{M,h}(-2m)=\Tr (\La_h^{2m}-D_h^{2m})=\Tr\left((a\La_g)^{2m}-(aD_g)^{2m} \right).
\end{equation*}
 Finally, using $\zeta_R(-1)=-\frac{1}{12}$, we obtain
\begin{equation*}
    \zeta_{M,h}(-1)=-\frac{\pi}{3}\sum_{i=1}^n \frac1{L_h(M_i)}+\Tr \left(a^{\frac12}\La_ga^{\frac12}-|a^{\frac12}D_{g}a^{\frac12}|\right).
\end{equation*} 
\end{proof}

\subsection{Conformal invariance of the  determinant}\label{remark}
By definition, the determinant of an operator is a spectral invariant. An important
feature of the DtN operator $\Lambda_g$ on a surface is that, after a normalization by the boundary length, the determinant ${\det}'\Lambda_g$ becomes conformally invariant. This was established by Guillarmou and Guillopé.  
\begin{proposition}[\cite{8180465}]\label{prop::zetaconformalinv}
    Let $\Omega$ be a compact connected surface with smooth boundary $M$. For any two conformally related metrics  $g$, $h$ on $\Omega$, we have
    \begin{equation}\label{eq::zetaconformalinv}
	\frac{{\det}' \La_g}{L_g(M)}=\frac{{\det}' \La_{h}}{L_{h}(M)}.
\end{equation}
\end{proposition}

Equivalently, the value of ${\det}' \La_g\big/L_g(M)$ depends only on the conformal class of the metric. This invariance allows us to analyze Steklov isospectral problems on surfaces through conformal transformations, thereby reducing the analysis to more convenient representatives within the same conformal class.

The original proof of Proposition \ref{prop::zetaconformalinv} utilizes the fact that  $\La_g$ can be identified with the scattering operator $S(1)$ associated with a conformally compact asymptotically hyperbolic surface (obtained by Mazzeo--Taylor uniformization \cite{Mazzeo2002}). The conformal variation of $\det{}'\Lambda_g$ is then computed by the Paycha--Scott formula for the Kontsevich--Vishik canonical trace \cite{PaychaScott2007,KontsevichVishik1995}.

\subsection{A  Polyakov-type formula and a proof of Proposition \ref{prop::zetaconformalinv}}

Since Proposition \ref{prop::zetaconformalinv} is crucial to this paper, we provide an alternative proof below for the sake of self-containedness. The proof is based on a more general Polyakov-type formula which may be of independent interest in other contexts.
 
Let $M$ be a closed $m$-dimensional manifold.Suppose that an operator $A_g$ on $M$ is associated with each Riemannian metric $g$. We say $A_g$ is {\bf conformally covariant of bidegree $(b_1, b_2)$} if    there exist real constants $b_1$ and $b_2$ such that for every $w\in C^\infty(M;\mathbb R)$,
\begin{equation*}
	A_{e^{2w}g}=e^{-b_1w}A_g e^{b_2w}.
\end{equation*} 
For a conformal path $g_t=e^{2w_t}g$, it is useful to introduce the
auxiliary metric
\begin{equation*}
\widehat g_t:=e^{2\frac{b_1+b_2}{m}w_t}g.
\end{equation*}
Then
\begin{equation*}
dV_{\widehat g_t}=e^{(b_1+b_2)w_t}dV_g.
\end{equation*}
Thus, if $A_g$ is self-adjoint on $L^2(M,g)$, then $A_{g_t}$ is
self-adjoint on $L^2(M,\widehat g_t)$.

In what follows,  we will fix
linearly independent functions $u_1^g,\ldots,u_N^g$ spanning $\ker A_g$.
By conformal covariance,
\begin{equation*}
\ker A_{g_t}=e^{-b_2w_t}\ker A_g.
\end{equation*}
Hence we set
\begin{equation*}
u_i^{g_t}:=e^{-b_2w_t}u_i^g,
\qquad i=1,\ldots,N.
\end{equation*}
These functions form a basis of $\ker A_{g_t}$. We define the corresponding
Gram matrix by
\begin{equation*}
G_{g_t}
:=
\big((u_i^{g_t},u_j^{g_t})_{L^2(M,\widehat g_t)}\big)_{i,j=1}^N.
\end{equation*}
Equivalently,
\begin{equation*}
(G_{g_t})_{ij}
=
\int_M e^{(b_1-b_2)w_t}u_i^g u_j^g\,dV_g.
\end{equation*}
  
  Now we may state the  Polyakov-type formula, which  is a mild zero-mode extension of the Paycha--Rosenberg conformal anomaly formula \cite[Corollary 2.6]{PaychaRosenberg2006} for invertible operators. In the case of the Laplace--Beltrami operator on closed surfaces, it recovers the classical Polyakov formula \cite{Polyakov1981}, as used for instance by Osgood, Phillips and Sarnak \cite{OSGOOD1988148}.
\begin{proposition}\label{prop::zero-mode-correction}
Let $A_g$ be a non-negative self-adjoint elliptic classical
pseudodifferential operator of order $\alpha>0$ on a closed $m$-dimensional
manifold $(M,g)$. Assume that the family $A_g$ is conformally covariant of
bidegree $(b_1,b_2)$. 
Then for any smooth path of metrics $g_t:=e^{2w_t}g$ connecting $g_0=g$ and $g_1=e^{2w}g$,  
\begin{equation*}
\log\frac{{\det}' A_{e^{2w}g}}{\det G_{e^{2w}g}}
- \log\frac{{\det}' A_g}{\det G_g}=\frac{b_1-b_2}{\alpha}
\int_0^1\int_M\dot w_t\operatorname{res}_x(\log A_{g_t})\,dt.
\end{equation*}
\end{proposition}

Here $\operatorname{res}_x$ denotes the local Guillemin--Wodzicki residue density
\cite{Guillemin1985,Wodzicki1987}; for the logarithmic case, see
\cite{Lesch1999,PaychaRosenberg2006}.

\begin{proof}
Since
\begin{equation*}
A_{g_t}=e^{-b_1w_t}A_g e^{b_2w_t},
\end{equation*}
differentiating gives
\begin{equation*}
\frac{dA_{g_t}}{dt}
=
-b_1\dot w_t A_{g_t}
+
b_2A_{g_t}\dot w_t .
\end{equation*}
We regard $A_{g_t}$ as a self-adjoint operator on $L^2(M,\widehat g_t)$.
Consider the spectral zeta function
\begin{equation*}
\zeta_{A_{g_t}}(s):=\operatorname{Tr}A_{g_t}^{-s}.
\end{equation*}
For $\operatorname{Re}s>0$ large,
\begin{equation*}
\frac{d}{dt}\zeta_{A_{g_t}}(s)
=
-s\operatorname{Tr}\left(
\frac{dA_{g_t}}{dt}A_{g_t}^{-s-1}
\right)=
-s(b_2-b_1)\operatorname{Tr}
\left(\dot w_tA_{g_t}^{-s}\right).
\end{equation*}
After meromorphic continuation and evaluation at $s=0$, this gives
\begin{equation*}
\frac{d}{dt}\log{\det}'A_{g_t}
=
(b_2-b_1)
\left.
\operatorname{Tr}\left(\dot w_tA_{g_t}^{-s}\right)
\right|_{s=0}^{\mathrm{mer}} .
\end{equation*}
Denote by $\Pi_{g_t}$ the $L^2(M,\widehat g_t)$-orthogonal projection onto
$\ker A_{g_t}$. Then by \cite[(0.14), or equivalently (0.15)]{PaychaScott2007}, we have
\begin{equation*}
\left.\operatorname{Tr}\left(\dot{w}_t A_{g_t}^{-s}\right)
\right|_{s=0}^{\mathrm{mer}}
=-\frac{1}{\alpha}\int_M \dot{w}_t\operatorname{res}_x\bigl(\log A_{g_t}\bigr)-\operatorname{Tr}(\dot{w}_t\Pi_{g_t}).
\end{equation*}
Combining the preceding identities yields
\begin{equation}\label{eq::varlogdet}
\frac{d}{dt}\log{\det}' A_{g_t}
=\frac{b_1-b_2}{\alpha}\int_M \dot{w}_t\operatorname{res}_x(\log A_{g_t})+(b_1-b_2)\operatorname{Tr}(\dot{w}_t\Pi_{g_t}).
\end{equation}

We now compute the derivative of the Gram determinant. By definition,
\begin{equation*}
(G_{g_t})_{ij}
=
\int_M e^{(b_1-b_2)w_t}u_i^g u_j^g\,dV_g.
\end{equation*}
Since $dV_{\widehat g_t}=e^{(b_1+b_2)w_t}dV_g$ we obtain
\begin{equation*}
\frac{d}{dt}(G_{g_t})_{ij}
=
(b_1-b_2)
(\dot w_tu_i^{g_t},u_j^{g_t})_{L^2(M,\widehat g_t)}.
\end{equation*}
It follows that
\begin{equation*}
\frac{d}{dt}\log\det G_{g_t}=\operatorname{Tr}\left(G_{g_t}^{-1}\frac{d}{dt}G_{g_t}\right) 
=
(b_1-b_2)
\sum_{i,j=1}^N
(G_{g_t}^{-1})_{ij}
(\dot w_tu_j^{g_t},u_i^{g_t})_{L^2(M,\widehat g_t)}.
\end{equation*}
On the other hand, the $L^2(M,\widehat g_t)$-orthogonal projection onto
$\ker A_{g_t}$ is given by
\begin{equation*}
\Pi_{g_t}f
=
\sum_{i,j=1}^N
u_i^{g_t}(G_{g_t}^{-1})_{ij}
(f,u_j^{g_t})_{L^2(M,\widehat g_t)} .
\end{equation*}
Therefore
\begin{equation*}
\operatorname{Tr}(\dot w_t\Pi_{g_t})
=
\sum_{i,j=1}^N
(G_{g_t}^{-1})_{ij}
(\dot w_tu_j^{g_t},u_i^{g_t})_{L^2(M,\widehat g_t)}.
\end{equation*}
Consequently,
\begin{equation*}
\frac{d}{dt}\log\det G_{g_t}
=
(b_1-b_2)\operatorname{Tr}(\dot w_t\Pi_{g_t}).
\end{equation*}
Combining this with \eqref{eq::varlogdet}, we obtain
\begin{equation*}
\frac{d}{dt}
\log\left(
\frac{{\det}'A_{g_t}}{\det G_{g_t}}
\right)
=
\frac{b_1-b_2}{\alpha}
\int_M\dot w_t\operatorname{res}_x(\log A_{g_t}).
\end{equation*}
Integrating from $t=0$ to $t=1$ completes the proof.
\end{proof}

\begin{remark}
The same proof applies to self-adjoint elliptic operators with finitely many negative eigenvalues, provided the determinant is defined by the zeta function associated with a fixed spectral cut on the non-zero spectrum.
\end{remark}

\begin{proof}[An alternative proof of Proposition \ref{prop::zetaconformalinv}]
For the DtN operator defined on the boundary $M$ of a surface, one has
\begin{equation*}
    \Lambda_{e^{2w}g}=e^{-w}\Lambda_g ,
\end{equation*}
so $b_1=1$, $b_2=0$, and $m=\dim M=1$. Moreover, by Lemma
\ref{symbols},
\begin{equation*}
    \operatorname{res}_x(\log \Lambda_{g_t}) = \frac12\operatorname{res}_x(\log D_{g_t}^2).
\end{equation*}
Since $D_{g_t}^2$ is a differential operator of order $2$, it is of
odd class. Hence, by the standard parity property for the classical part of the logarithm of an even-order odd-class elliptic operator (see, for instance, \cite[Lemma 2.13]{Braverman2009} and \cite[Lemma 2.3]{Guillarmou2009Krein}),
$\sigma_{-1}(\log D_{g_t}^2)$ is odd in $\xi$, and therefore its integral over
the unit cosphere vanishes. Thus $\operatorname{res}_x(\log D_{g_t}^2)=0$, and consequently
$\operatorname{res}_x(\log \Lambda_{g_t})=0$.
So Proposition \ref{prop::zero-mode-correction} implies that
\[
\frac{{\det}'\Lambda_g}{\det G_g}
\]
is conformally invariant.
Moreover $\ker\Lambda_g$ consists  of the constant functions. Taking the basis $1$ of the kernel gives
\[
G_g=\big(( 1,1)_{L^2(M,g)}\big)=(L_g(M)).
\]
Therefore Proposition
\ref{prop::zetaconformalinv} follows. 
\end{proof}

For conformally covariant elliptic differential operators on closed odd-dimensional manifolds, the conformal invariance  of the
determinant was proved by Rosenberg \cite{Rosenberg1987}.  See also Parker and  Rosenberg \cite{ParkerRosenberg1987} for the conformal Laplacian. For scattering operators $S(\lambda)$ in the even asymptotically hyperbolic setting with
odd-dimensional conformal infinity, Guillarmou proved the corresponding conformal invariance of the scattering determinant $\det S(\lambda)$ \cite{Guillarmou2009Krein}.

\section{Steklov spectral analysis on flat cylinders}\label{Section::OPS}

\subsection{The  Osgood--Phillips--Sarnak uniformization of annular surfaces}\label{uniformazation}

In their study of the extremal properties of the log-determinant of the Laplacian, Osgood, Phillips, and Sarnak established a new uniformization theorem \cite{OSGOOD1988148}, which classifies conformal classes of compact surfaces with non-positive Euler characteristic, both with and without boundary. 
Specifically, they proved that within any conformal class of Riemannian metrics on a compact surface with a non-positive Euler characteristic, the following uniform metrics exist uniquely up to multiplication by a constant:
\begin{enumerate}[(1)]
	\item 
For a closed surface, there is a unique constant curvature metric in that class.
\item For a surface with boundary, there are two distinct types of uniform metrics:
\begin{enumerate}[(i)]
	\item (Type I) a constant curvature metric with geodesic boundary;
	\item (Type II) a flat metric whose boundary has constant geodesic curvature.
\end{enumerate} 
\end{enumerate} 

Recall that $\mathcal{A}$ denotes the set of smooth annular surfaces. A standard model for $\mathcal{A}$ is given by the cylinders
\begin{equation}\label{el}
	\mathcal{C}_l:=S^1 \times [0,l] = \{(e^{i\theta},t)|\ \theta\in[0,2\pi),t\in[0,l]\}
\end{equation} 
with boundary $\Sigma=\Sigma_1 \cup \Sigma_2$, equipped with  the standard product metric  $g_0=dt^2+d\theta^2$.
 Note that $g_0$ itself can be regarded as both a type I and a type II uniform metric.

Each surface $S\in \mathcal{A}$ has vanishing Euler characteristic, i.e. $\chi(S)=0$. In view of the Gauss--Bonnet theorem and the  Osgood--Phillips--Sarnak  uniformization theorem mentioned above,  $S$ is conformally equivalent to a standard cylinder $(\mathcal{C}_l,g_0)$ for some parameter $l$.  In other words, one may  identify $S$ with  $(\mathcal{C}_l,g)$ for some
\begin{equation}\label{conformalmetricform} g=e^{2\phi}g_0, \phi\in C^\infty(\mathcal{C}_l;\mathbb{R}).
\end{equation}
For such a conformal metric $g$, the Gauss curvature $K_g$ of $\mathcal{C}_l$ is given by 
\begin{equation*}
	K_g=e^{-2\phi}(-\Delta_{g_0}\phi+K_{g_0}), 
\end{equation*}
while   the geodesic curvature $k_g$   of the boundary is 
\begin{equation*}
	k_g=e^{-\phi}(\partial_n\phi+k_{g_0}).
\end{equation*}

In particular, if $(\mathcal{C}_l,g)$ is flat and has geodesic boundary, then necessarily the conformal factor $\phi$ is a constant. Consequently, cylinders $(\mathcal{C}_l, g_0)$ with different values of $l$ belong to distinct conformal classes.

Moreover, let $\mathrm{Aut}(\mathcal{C}_l)$ denote the conformal automorphism group of $(\mathcal{C}_l,g_0)$. This group is rigid in the sense that every conformal self-map is an isometry; namely,
\[\mathrm{Aut}(\mathcal{C}_l)=\mathrm{Iso}(\mathcal{C}_l)\cong O(2)\times \mathbb Z_2.
\]
This differs from the simply connected case, where the unit disk admits a rich family of nontrivial conformal self-maps (M\"obius transformations). Indeed,
\[
\mathrm{Aut}(\mathbb{D})\cong \mathrm{PSL}(2,\mathbb{R}),
\]
which is non-compact.
This phenomenon in the annular setting explains why, unlike in \cite{JOLLIVET20181712}, no additional normalization procedure modulo certain conformal automorphisms is required in the compactness analysis in Section \ref{compactness}.

The following lemma quantifies such a  rigidity  via the first nonzero Steklov eigenvalue. Let $dV_g$ and $d\theta_g$ denote the area measure on $\mathcal{C}_l$ and the induced line measure on $\Sigma=\partial\mathcal{C}_l$ with respect to the metric $g$, respectively. Similarly, let $dV$ and $d\theta$ denote the corresponding measures associated with the flat metric $g_0$. Then
\[
dV_g = e^{2\phi}dV, \qquad d\theta_g = e^{\phi}d\theta.
\]
\begin{lemma}\label{rigidityphi}
     Let $g$ be a flat metric on $\mathcal{C}_l$ given by \eqref{conformalmetricform}. Denote by $\sigma_1:=\sigma_1(\Lambda_g,\Sigma)$ the first nonzero Steklov eigenvalue of $(\mathcal{C}_l,g)$. Then the Dirichlet energy of the conformal factor $\phi$ satisfies
    \begin{equation*}
    \int_{\mathcal{C}_l}|\nabla_{g_0} \phi|^2dV\leq\frac{1}{\sigma_1}\int_\Sigma k_g^2d\theta_g.
\end{equation*}
\end{lemma}
\begin{proof} 
    Since $g$ is flat, $\phi$ is a $g_0$-harmonic function and satisfies
    \[
    \d_{n(g)}\phi=k_g,\quad \text{on }\Sigma,
    \]
    where $n(g)$ is the outward normal vector with respect to the metric $g$, and we used the fact $e^{-\phi}\partial_n \phi =   \d_{n(g)}\phi$.
    According to the fact
    \[
    (\Lambda_{g}\phi,\phi)_{L^2(\Sigma,g)}=
(\Lambda^{\frac12}_{g}\phi,\Lambda^{\frac12}_{g}\phi)_{L^2(\Sigma,g)}
     \leq  \frac{1}{\sigma_1} (\Lambda^{\frac32}_{g}\phi,\Lambda^{\frac12}_{g}\phi)_{L^2(\Sigma,g)}
     =\frac{1}{\sigma_1}\|\Lambda_g\phi\|_{L^2(\Sigma,g)}^2
    \]
    and the conformal invariance of the Dirichlet energy in dimension two, we have
    \[
   \int_{\mathcal{C}_l}|\nabla_{g_0} \phi|^2dV= \int_{\mathcal{C}_l}|\nabla_{g} \phi|^2dV_g=(\Lambda_{g}\phi,\phi)_{L^2(\Sigma,g)}\leq\frac{1}{\sigma_1}\|\Lambda_g\phi\|_{L^2(\Sigma,g)}^2=\frac{1}{\sigma_1}\int_\Sigma k_g^2d\theta_g.
    \]
\end{proof}
As an application of the preceding lemma, we obtain the following estimate for the conformal factor, which is in
the spirit of the a priori estimates in \cite{ChangYang1989,ChangYang}, but is formulated here in the Steklov setting for flat annular surfaces.

\begin{proposition}\label{H1normupbd}
Let $g$ be a flat metric on $\mathcal{C}_l$ given by \eqref{conformalmetricform}. Denote by $\sigma_1:=\sigma_1(\Lambda_g,\Sigma)$ the first nonzero Steklov eigenvalue of $(\mathcal{C}_l,g)$. Then there exists a constant $C(l)>0$, depending only
on $l$, such that
\[
\|\phi\|_{H^1(\mathcal C_l,g_0)}
\leq C(l)\left[\left(\frac1{\sigma_1}\int_\Sigma k_g^2\,d\theta_g
\right)^{\frac12}+\frac1{\sigma_1}\int_\Sigma k_g^2\,d\theta_g+\left|\log\left(\frac{L_g(\Sigma)}{4\pi}\right)\right|
\right].
\]
\end{proposition}
\begin{proof}
Decompose $
\phi=\psi-c$, where $c\in\mathbb R$  is chosen so that
\[
\int_\Sigma \psi\,d\theta=0.
\]
Since
\[
L_g(\Sigma)=\int_\Sigma e^\phi\,d\theta
=e^{-c}\int_\Sigma e^\psi\,d\theta,
\]
we obtain
\begin{equation}\label{eq::cdef}
   c=\log\left(\frac1{4\pi}\int_\Sigma e^\psi\,d\theta\right)-\log\left(\frac{L_g(\Sigma)}{4\pi}\right).
\end{equation}
By Jensen's inequality, we get \[\log\left(\frac1{4\pi}\int_\Sigma e^\psi\,d\theta\right)\geq 0.\] 

For $j=1,2$, we set
\[
\mathrm{Aver}_{\Sigma_j,g_0}(\psi):=\frac1{2\pi}\int_{\Sigma_j}\psi\,d\theta,
\qquad
v_j:=\psi|_{\Sigma_j}-\mathrm{Aver}_{\Sigma_j,g_0}(\psi).
\]
Then
\[
\int_{\Sigma_j}v_j\,d\theta=0.
\]
Since \(\int_\Sigma \psi\,d\theta=0\),  writing $r=\mathrm{Aver}_{\Sigma_1,g_0}(\psi)=-\mathrm{Aver}_{\Sigma_2,g_0}(\psi)$, we have
\[
\psi|_{\Sigma_1}=v_1+r,\qquad
\psi|_{\Sigma_2}=v_2-r.
\]
Note that by definition and Cauchy--Schwarz inequality, one has $r^2 \le \frac 1{2\pi}\int_{\Sigma_i}\psi^2$ for each $i$, and thus $r^2 \le \frac 1{4\pi}\|\psi\|_{L^2(\Sigma)}^2$. 

On each boundary component $\Sigma_j$, which is isometric to a unit circle, the
Beckner--Onofri inequality (see \cite{Beckner1993,OSGOOD1988148}) gives
\[
\log\left(
\frac1{2\pi}\int_{\Sigma_j}e^{v_j}\,d\theta\right)\leq\frac{1}{4\pi}\left\| |D_{g_0}|^{\frac12}v_j\right\|^2_{L^2(\Sigma_j,g_0)},
\qquad j=1,2.
\]
Therefore

\begin{align*}
\frac1{4\pi}\int_\Sigma e^\psi\,d\theta
&=\frac12 e^r
\left(\frac1{2\pi}\int_{\Sigma_1}e^{v_1}\,d\theta
\right)+\frac12 e^{-r}
\left(\frac1{2\pi}\int_{\Sigma_2}e^{v_2}\,d\theta
\right)\notag \\
&\leq \frac12 e^r
\exp\left(\frac{1}{4\pi}\left\| |D_{g_0}|^{\frac12}v_1\right\|^2_{L^2(\Sigma_1,g_0)}\right)
+\frac12 e^{-r}\exp\left(\frac{1}{4\pi}\left\| |D_{g_0}|^{\frac12}v_2\right\|^2_{L^2(\Sigma_2,g_0)}\right)\notag\\
&\leq\cosh(r)\,\exp\left(
\frac{1}{4\pi}\sum_{j=1}^2\left\| |D_{g_0}|^{\frac12}v_j\right\|^2_{L^2(\Sigma_j,g_0)}\right).
\end{align*}
Taking logarithms and using the elementary bound
\[
\log\cosh(r)\leq \frac12 r^2,
\]
we get
\begin{equation}\label{eq::estimateforc}
\log\left(\frac1{4\pi}\int_\Sigma e^\psi d\theta
\right)
\leq\frac12 r^2+\frac{1}{4\pi}\sum_{j=1}^2\left\| |D_{g_0}|^{\frac12}v_j\right\|^2_{L^2(\Sigma_j,g_0)}.
\end{equation}
The right-hand side is controlled by the $L^2$-norm of $\nabla_{g_0}\psi$. Indeed,
since $\psi$ has zero boundary average, the trace estimate and the Poincaré inequality give
\begin{equation}\label{highestimateforc}
\sum_{j=1}^2\left\| |D_{g_0}|^{\frac12}v_j\right\|^2_{L^2(\Sigma_j,g_0)}\leq \|\psi\|_{H^{\frac12}(\Sigma,g_0)}^2\leq C(l)\|\nabla_{g_0}\psi\|^2_{L^2(\mathcal C_l,g_0)}
=C(l)\|\nabla_{g_0}\phi\|^2_{L^2(\mathcal C_l,g_0)}.
\end{equation}
For the first term, we have 
\begin{equation} \label{estimateforcloow}
    r^2\le\frac{1}{4\pi}\|\psi\|^2_{L^2(\Sigma,g_0)}
    \leq C(l)\|\nabla_{g_0}\psi\|^2_{L^2(\mathcal C_l,g_0)}
=C(l)\|\nabla_{g_0}\phi\|^2_{L^2(\mathcal C_l,g_0)}.
\end{equation}
Combining \eqref{eq::cdef}, \eqref{eq::estimateforc}, \eqref{highestimateforc} and \eqref{estimateforcloow}, we obtain
\[
|c|\leq C(l)\|\nabla_{g_0}\psi\|^2_{L^2(\mathcal C_l,g_0)}+\left|\log\left(\frac{L_g(\Sigma)}{4\pi}\right)\right|.
\]
Finally,
\[
\|\phi\|_{H^1(\mathcal C_l,g_0)}\leq \|\psi\|_{H^1(\mathcal C_l,g_0)}+C(l)|c|.
\]
Similarly, since $\psi$ has zero boundary average, it follows from the Poincaré inequality that
\[
\|\psi\|_{H^1(\mathcal C_l,g_0)}\leq C(l)\|\nabla_{g_0}\psi\|_{L^2(\mathcal C_l,g_0)}=
C(l)\|\nabla_{g_0}\phi\|_{L^2(\mathcal C_l,g_0)}.
\]
Combining this with the previous estimate for $c$, we get
\[
\|\phi\|_{H^1(\mathcal C_l,g_0)}
\leq C(l)\left(
\|\nabla_{g_0}\phi\|_{L^2(\mathcal C_l,g_0)}+\|\nabla_{g_0}\phi\|_{L^2(\mathcal C_l,g_0)}^2
+\left|\log\left(\frac{L_g(\Sigma)}{4\pi}\right)\right|\right).
\]
The desired result then follows from Lemma \ref{rigidityphi}.
\end{proof}

Combining Proposition \ref{H1normupbd} with Rellich's compact embedding theorem, we obtain the following compactness result which can be viewed as the Steklov analogue of the theorem in the appendix of  \cite{ChangYang}. 
\begin{corollary}\label{compactinconfclass}
   Let $\mathcal{F}_{L,\kappa,\bar{\sigma}}\subset \mathcal{A}$ be a family of
flat annular surfaces in a single conformal class. Suppose that for each $S\in\mathcal{F}_{L,\kappa,\bar{\sigma}}$, the following conditions hold:
    \begin{enumerate}
        \item the total length of $\d S$ equals $L$;
        \item the geodesic curvature $k_{\d S}$ of $\d S$ satisfies $\|k_{\d S}\|_{L^2}\leq \kappa$;
        \item the first nonzero Steklov eigenvalue $\sigma_1(\Lambda_{\d S},\d S)\geq \bar{\sigma}>0$.
    \end{enumerate}
    Then $\mathcal{F}_{L,\kappa,\bar{\sigma}}$ is compact in the $H^{1-\epsilon}$ topology for any $\epsilon>0$.
\end{corollary}

\subsection{The Steklov determinant on flat cylinders} 
We begin by recalling the Steklov eigenvalue problem on a flat cylinder $\mathcal{C}_l$ of height $l$, equipped with the standard product metric $g_0$. Recall that $\Sigma=\Sigma_1\cup\Sigma_2$ denotes the boundary of $\mathcal{C}_l$.
For convenience, we write \[D:=D_{\Sigma,g_0},\quad \La:=\La_{\Sigma,g_0}.\]

We shall normalize the inner product on $L^2(\Sigma)$ as  
\begin{equation}\label{innerprod}
    (f,g)_{L^2(\Sigma)}:=\frac{1}{4\pi}\int_\Sigma f\bar{g}\ d\theta.
\end{equation}
Set  
\begin{equation}\label{DtNeiforcyl}
    \sigma_k^{(1)}=k\tanh \frac{kl}{2},  \quad \text{and}\quad 
    \sigma_k^{(-1)} = \left\{
    \begin{array}{ll}
    	k\coth\frac{kl}{2}, & k \ne 0,\\
    2/l, & k=0.
    \end{array}
    \right.
\end{equation}   
Then  the Steklov eigenvalues of $(\mathcal{C}_l,g_0)$ are precisely the numbers $ \sigma_k^{(1)}$ and $\sigma_k^{(-1)}$ ($k \in \mathbb Z$), with corresponding $L^2$-normalized eigenfunctions 
\begin{equation}\label{cyleigenfunc}
    \psi_k^{(1)}(\theta)=\left\{\begin{aligned}
      &e^{\sqrt{-1}k\theta},  & \text{on }\Sigma_1\\     
      &e^{\sqrt{-1}k\theta},   & \text{on }\Sigma_2
    \end{aligned}\right.
    \quad\text{and}\quad
    \psi_k^{(-1)}(\theta)=\left\{\begin{aligned}
      &-e^{\sqrt{-1}k\theta},   &\text{on }\Sigma_1\\     
      &e^{\sqrt{-1}k\theta},       &\text{on }\Sigma_2
    \end{aligned}\right.
\end{equation}
respectively. Namely, we have
\begin{equation*}
    \La \psi^{(\nu)}_k=\sigma^{(\nu)}_k \psi^{(\nu)}_k.
\end{equation*}

Note that in this special example, the product of two normalized eigenfunctions remains a normalized eigenfunction,  namely, 
\begin{equation}\label{prodrule}
    \psi_k^{(\eta)}\cdot \psi_j^{(\nu)}=\psi_{k+j}^{(\eta\nu)},
\end{equation}
for each $k, j\in\Z$ and $\eta, \nu\in \Z_2^\times=\{1,-1\}$.

With these eigenvalues at hand, we can directly calculate the log-determinant of  $\La$. 

\begin{proposition}\label{prop:cylinderlogdet}
 We have
    \begin{equation}\label{eq::2.8}
        \log{\det}' \Lambda=-\zeta_{\Sigma,g_0}'(0)=\log \frac{2}{l}+2 \log (2\pi).
    \end{equation}
\end{proposition}
\begin{proof}
Let $\zeta_R(s)$ be the Riemann zeta function. Then by \eqref{DtNeiforcyl},
\begin{equation}\label{eq::zetacylinder}
	4\zeta_R(s)  \!-\! \zeta_{\Sigma,g_0}(s)\!=\!\!\sum_{k\in\Z\setminus\{0\}}\!|k|^{-s}\left(2\!-\!\left(\tanh \frac{|k|l}{2}\right)^{-s}\!-\!\left(\coth \frac{|k|l}{2}\right)^{-s}\right)\!-\!\left(\frac{2}{l}\right)^{-s}
\end{equation}
for all $s\in \C$.  Note that as $x\to \infty$,
\begin{equation*}
\tanh x=1-O(x^{-\infty}),\quad \text{and}\quad \coth x=1+O(x^{-\infty}).
\end{equation*}
So we may take the derivative of \eqref{eq::zetacylinder} term by term and let $s=0$, and get
\[   \log{\det}' \Lambda=-\zeta_{\Sigma,g_0}'(0)=\log \frac{2}{l}-4\zeta_R'(0).\]
Now the result follows from the fact $\zeta_R'(0)=-\frac 12 \log (2\pi)$.
\end{proof}
\begin{corollary}\label{specdetcon}
  Let $S\in\mathcal{A}$ be an annular surface with boundary $M$. Then the Steklov spectrum of $S$ determines its conformal class.
\end{corollary}
\begin{proof}
    According to the discussion in Section \ref{uniformazation}, $\Omega$  is conformal to some standard cylinder $(\mathcal{C}_l,g_0)$ and we can identify  $\Omega$ with $(\mathcal{C}_l,g)$, where $g$ is conformal to $g_0$. Moreover, the conformal class of $\Omega$ is uniquely characterized by $l$. By Proposition \ref{prop::zetaconformalinv} and Proposition \ref{prop:cylinderlogdet}, we have
    \[
    \frac{{\det}' \La_g}{L_g(\Sigma)}=\frac{{\det}' \La}{L_{g_0}(\Sigma)}=\frac{ e^{\log\frac{2}{l}+2 \log (2\pi)}}{4\pi}=\frac{2\pi}{l},
    \]
    which implies
\[
l= \frac{2\pi L_{g}(\Sigma)}{{\det}' \La_g}.
\]
This completes the proof.
\end{proof}
\begin{remark}
	One may apply the same argument to the M\"obius band  $\mathbb M$, and show that the Steklov spectrum associated with any Riemannian metric on $\mathbb M$ determines its conformal class. To see this, one may start with a flat cylinder and glue antipodal points on one boundary component. With this specific surface as model surface, one can explicitly compute its Steklov eigenvalues (again by separation of variables), and then calculate  the corresponding log-determinant. However, the zeta value $\zeta(-2)$ for the  M\"obius band  is no longer a summation of nonnegative terms as in the case of the  cylinder, and thus one cannot apply the same arguments below to prove an analogous compactness result. 
\end{remark}

\section{Zeta invariants and inverse spectral problems}\label{sec::inv}
\subsection{Steklov zeta invariants for annular surfaces}
Let $S\in \mathcal A$ be an annular surface represented by
$(\mathcal C_l,g)$, where 
$g=e^{2\phi}g_0$ 
for some real-valued function $\phi\in C^\infty(\mathcal C_l)$. In what follows, we compute the zeta invariants of $S$, namely the values of the spectral zeta function $\zeta_{\partial S}$ associated with the Dirichlet-to-Neumann operator $\Lambda_g$ at negative even integers.

Define 
\begin{equation}\label{def::a}
	a:=e^{-\phi}|_{\Sigma} \in C^\infty(\Sigma)
\end{equation}
to be a positive function on $\Sigma$.
Then 
\begin{equation*}
	\La_g=a\Lambda,\quad D_g=aD.
\end{equation*}

Note that the eigenfunctions $\{\psi_k^{(1)}\}\cup\{\psi_j^{(-1)}\}$ of $\Lambda$ also form the complete set of eigenfunctions of the operators $D$ and  $|D|$. Namely,
\begin{equation*}
    D\psi_k^{(\nu)}=k\psi_k^{(\nu)},\quad |D|\psi_k^{(\nu)}=|k|\psi_k^{(\nu)},
\end{equation*}
for each $k\in\Z$ and $\nu\in\Z_2^\times$. Therefore the operators $\Lambda$ and $D$ commute.

From the discussion in Section \ref{defzetainvariant} and the equalities \eqref{generalzeta-1}, \eqref{generalzeta-2m}, the zeta invariants of the Steklov spectrum of $\Omega$ can be represented as
\begin{equation}\label{eq::zetainvarcal}
    \zeta_{\Sigma,g}(-2m)=\Tr \left((a\Lambda)^{2m}-(aD)^{2m}\right), \quad m\in\Z_+,
\end{equation}
and we also have
 \begin{equation}\label{eq::zeta2-1}
     \zeta_{\Sigma,g}(-1)=-\frac{\pi}{3}\sum_{i=1}^2\frac{1}{L_g(\Sigma_i)}+\Tr (a^{\frac12}\Lambda a^{\frac12}-|a^{\frac12}Da^{\frac12}|).
 \end{equation}
 
 In the case where the domain $\Omega$ is simply connected, these quantities have been explicitly expressed through the Fourier coefficients of $a$ in \cite{Malkovich2015}. Here we will compute these zeta invariants for an annular surface $\Omega$ in a similar manner.
The positive function $a\in C^\infty(\Sigma)$ defined in \eqref{def::a} can be written as
\begin{equation}\label{eq::expansionofa}
    a=\sum_{\nu=\pm1}\sum_{k\in\Z}a_k^{(\nu)}\psi_k^{(\nu)},
\end{equation}
where
\begin{equation}\label{eq::defofak}
    a_k^{(\nu)}=\left(a,\psi_k^{(\nu)}\right)_{L^2(\Sigma)} 
\end{equation} are the Fourier coefficients of $a$. In particular, by our normalization \eqref{innerprod} and \eqref{cyleigenfunc},
\[a_0^{(1)}=\frac{1}{4\pi}\int_{\Sigma} a\ d\theta. 
\]

Now we express the zeta invariants $\zeta_{\Sigma,g}(-2m)$ via these Fourier coefficients. 
\begin{proposition}
    The zeta invariants $\zeta_{\d S}(-2m)=\zeta_{\Sigma,g}(-2m)$ ($m\in\Z_+$)  of an annular surface $S\in\mathcal{A}$ 
are given by 
    \begin{align}
\zeta_{\Sigma,g}(-2m)&=\Tr \left((a\Lambda)^{2m}-(aD)^{2m}\right)\notag\\
&=\sum_{\substack{{\nu_1\nu_2\dots\nu_{2m}}=1\\ \nu_1,\dots,\nu_{2m}\in \Z_2^\times}} \sum_{\substack{k_1+k_2+\cdots k_{2m}=0\\ k_1,\dots,k_{2m}\in\Z}}N_{k_1,k_2,\cdots,k_{2m-1}}^{\nu_1,\nu_2,\cdots,\nu_{2m-1}}a_{k_1}^{(\nu_{1})}a_{k_2}^{(\nu_{2})}\cdots a_{k_{2m}}^{(\nu_{2m})},\label{eq::zetafunctionexpansion}
    \end{align}
    where the coefficients are  
    \begin{align}
        N_{k_1,k_2,\cdots,k_{2m-1}}^{\nu_1,\nu_2,\cdots,\nu_{2m-1}}=&\sum_{j\in\Z}\Big(\sigma_j^{(1)}\sigma_{j+k_1}^{(\nu_1)}\sigma_{j+k_1+k_2}^{(\nu_1\nu_2)}\cdots\sigma_{j+k_1+k_2+\dots k_{2m-1}}^{(\nu_1\nu_2\cdots\nu_{2m-1})}\notag\\ & +  
        \sigma_j^{(-1)}\sigma_{j+k_1}^{(-\nu_1)}\sigma_{j+k_1+k_2}^{(-\nu_1\nu_2)}\cdots\sigma_{j+k_1+k_2+\dots k_{2m-1}}^{(-\nu_1\nu_2\cdots\nu_{2m-1})}\notag\\
        &-2j(j+k_1)(j+k_1+k_2)\cdots(j+k_1+k_2+\cdots+k_{2m-1})
        \Big).
    \end{align}
\end{proposition}
\begin{proof}
    The proof is based on a direct calculation. According to  \eqref{eq::zetainvarcal}, 
\begin{equation}\label{eq::zeta-2mcal}
    \zeta_{\Sigma,g}(-2m)%=\Tr \left((a\La)^{2m}-(aD)^{2m}\right)
    =\sum_{\nu\in\{1,-1\}}\sum_{j\in\Z}\left(\left((a\Lambda)^{2m}-(aD)^{2m}\right) \psi_j^{(\nu)}, \psi_j^{(\nu)}\right)_{L^2(\Sigma)}.
\end{equation}
Substituting \eqref{eq::expansionofa} into \eqref{eq::zeta-2mcal}, and using the product rule \eqref{prodrule} iteratively, the inner product $\left((a\Lambda)^{2m} \psi_j^{(\nu)}, \psi_j^{(\nu)}\right)_{L^2(\Sigma)}$ becomes a summation of terms of the form % Then we complete the proof by noting that each term
    \begin{align}
        \Big(\psi_{k_{2m}}^{(\nu_{2m})}\Lambda\psi_{k_{2m-1}}^{(\nu_{2m-1})}&\Lambda\cdots\psi_{k_1}^{(\nu_{1})}\Lambda\psi_j^{(\nu)},\psi_j^{(\nu)}\Big)_{L^2(\Sigma)}
        \notag\\
&=\sigma_j^{(\nu)}\sigma_{j+k_1}^{(\nu\nu_1)}\cdots\sigma_{j+k_1+k_2+\dots k_{2m-1}}^{(\nu\nu_1\nu_2\cdots\nu_{2m-1})}\left(\psi_{j+k_1+k_2+\cdots+k_{2m}}^{(\nu\nu_1\nu_2\cdots\nu_{2m})}, \psi_j^{(\nu)}
        \right)_{L^2(\Sigma)},
    \end{align}
    which vanishes unless 
\begin{equation*}
    \nu_1\nu_2\cdots\nu_{2m}=1, \quad\text{and }\quad k_1+k_2+\cdots+k_{2m}=0.
\end{equation*}
The inner product $\left((aD)^{2m} \psi_j^{(\nu)}, \psi_j^{(\nu)}\right)_{L^2(\Sigma)}$ can be calculated in the same way.
\end{proof}

Since $a$ is a positive real-valued function, the coefficients in \eqref{eq::expansionofa} satisfy  
\begin{equation*}
    a_k^{(\nu)}=\overline{{a_{-k}^{(\nu)}}}.
\end{equation*}
In particular, when $m=1$ the formula \eqref{eq::zetafunctionexpansion} can be written as
\begin{align}\label{eq::zeta-2}
    \zeta_{\Sigma,g}(-2)=\sum_{k\in\Z} \left(B_l(k)\left|a_k^{(1)}\right|^2
    +B_l'(k)\left|a_k^{(-1)}\right|^2\right).
\end{align}
The constants are given by
\begin{equation}\label{eq::zeta-2Ck}
    B_l(k)=\sum_{j\in\Z }\left(\sigma_{j}^{(1)}\sigma_{j+k}^{(1)}+\sigma_{j}^{(-1)}\sigma_{j+k}^{(-1)}-2j(j+k)\right).
\end{equation}
By the AM--GM inequality, each term satisfies $\sigma_{j}^{(1)}\sigma_{j+k}^{(1)} +\sigma_{j}^{(-1)}\sigma_{j+k}^{(-1)}> 2j(j+k)$. Hence $B_l(k)>0$ for each $k\in\Z$. 
In particular, 
\begin{equation}
    B_l(0)=\sum_{j\in\Z\setminus\{0\} }\left(j^2\tanh^2\frac{jl}{2}+j^2\coth^2\frac{jl}{2}-2j^2\right)+\frac{4}{l^2}>0.
\end{equation}

Similarly, we have 
\[B_l'(0)=\sum_{j\in\Z\setminus\{0\} }\left(2j^2\tanh\frac{jl}{2}\coth\frac{jl}{2}-2j^2\right)=0\]
and for $k \ne 0$,
\begin{equation}\label{eq::zeta-2Ck'}
    B_l'(k)=\sum_{j\in\Z }\left(\sigma_{j}^{(1)}\sigma_{j+k}^{(-1)}+\sigma_{j}^{(-1)}\sigma_{j+k}^{(1)}-2j(j+k)\right)> 0.
\end{equation}

\subsection{Applications to Steklov inverse spectral problems}\label{application}

Let $A_{r_1, r_2,h}\in \mathcal{A}$ be the lateral surface of a truncated cone of height $h$ whose boundary circles have lengths $r_1, r_2$. More precisely, for $h>0$ and $r_1\leq r_2$, let 
\[
A_{r_1, r_2,h} := \left\{(x,y,z)\in\R^3\ |\ x^2+y^2=\left(\frac{r_2-r_1}{2\pi h}z+\frac{r_1}{2\pi}\right)^2, \ z\in[0,h]\right\},
\] 
while for $h=0$ and $r_1<r_2$, let
\[
A_{r_1, r_2,0} := \left\{ (x,y) \mid \frac{r_1^2}{4\pi^2}\leq x^2+y^2 \leq \frac{r_2^2}{4\pi^2}\right\}.
\]
Note that the case $r_1=r_2$ gives a cylinder, while $h=0$ corresponds to a planar annulus. It is not hard to show $a_k^{(1)}=a_k^{(-1)}=0$ for all $k \ne 0$, and hence \eqref{eq::zeta-2} gives
\[ 
\zeta_{\d A_{r_1, r_2,h}}(-2)=\pi^2B_{l_A}(0)\left(\frac{1}{r_1}+\frac{1}{r_2}\right)^2.
\] 
To calculate $l_A$,  we denote by $(\mathcal{C}_{l_A},e^{2u}g_0)$  the cylinder that is isometric to $A_{r_1, r_2,h}$. Then % is isometric to a cylinder  for some $l_A$ and $u\in C^\infty(E_{l_A},\R)$.I
in the coordinates of \eqref{el}, $u$ can be expressed as
\begin{equation}\label{exprofu}
    u(\theta,t)=\frac{t}{l_A}\log \frac{r_2}{r_1}+\log \frac{r_1}{2\pi}.
\end{equation}
Hence the length of the generatrix of $A_{r_1, r_2,h}$ is given by
\[
\sqrt{h^2+\left(\frac{r_2}{2\pi}-\frac{r_1}{2\pi}\right)^2}=\int_0^{l_A}e^udt=\begin{cases}\frac{l_A}{\log\frac{r_2}{r_1}}\left(\frac{r_2}{2\pi}-\frac{r_1}{2\pi}\right),& r_1<r_2,\\
\frac{r_1}{2\pi}l_A, &r_1=r_2.
\end{cases}
\]
It follows that
\begin{equation}\label{exproflA}
    l_A=\begin{cases}\log\frac{r_2}{r_1}\sqrt{1+\frac{(2\pi h)^2}{(r_2-r_1)^2}}, &r_1<r_2,\\
    \frac{2\pi h }{r_1},&r_1=r_2.
    \end{cases}
\end{equation}

As a consequence of \eqref{eq::zeta-2}, we prove 
\begin{theorem}\label{thm::compactnessdoubly}
	For any $r_1 \le r_2$ and $h \ge 0$,   $A_{r_1, r_2,h}$  minimizes $\zeta_{\d S}(-2)$ among the set of surfaces 
	\[\aligned
	\mathcal A_{r_1, r_2, h} := \{S\ |\  S\in\mathcal{A} &\text{ has two boundary components of lengths } r_1, r_2, \\
    &\text{ and is conformal to the cylinder } (\mathcal{C}_{l_A}, g_0) \}
	\endaligned\]
	(where $l_A$ is given by \eqref{exproflA}), and is the unique minimizer of  $\zeta_{\partial S}(-2)$ in the set
	\[
	\mathcal A_{r_1, r_2, h}^{\text{\tiny flat}} := \{S \in \mathcal A_{r_1, r_2, h}\ |\  S\text{ is flat }\}.
	\]
	\end{theorem}
\begin{proof} According to the uniformization theorem, we may identify $S$ with $(\mathcal{C}_{l_A},e^{2\phi}g_0)$ for some $l_A$ and $\phi\in C^\infty(\mathcal{C}_{l_A};\R)$. Write $\partial S=\Sigma=\Sigma_1 \cup \Sigma_2$ with $L_g(\Sigma_i)=r_i$ ($i=1, 2$) and denote $a=e^{-\phi}|_{\Sigma}$.
	
By \eqref{eq::zeta-2}, %for any $(\Omega,g) \in 	\mathcal A_{r_1, r_2, h}$, 
the value of $\zeta_{\Sigma,g}(-2)$ has a lower bound
\begin{equation}\label{ineq1}
  \zeta_{\Sigma,g}(-2)\geq B_{l_A}(0)\left|a_0^{(1)}\right|^2=\frac{1}{16\pi^2}B_{l_A}(0)\left(\int_{\Sigma} a\ d\theta\right)^2.
\end{equation}
On the other hand, since 
\begin{equation}
	r_i=L_g(\Sigma_i)=\int_{\Sigma_i}a^{-1}d\theta,\quad i=1,2,
\end{equation}
by Cauchy--Schwarz inequality, 
\begin{equation}
	\int_{\Sigma_i} a\ d\theta \ge \frac{4\pi^2}{r_i}, \quad i=1,2.
\end{equation} 
It follows that for any $(\Omega, g) \in 	\mathcal A_{r_1, r_2, h}$, 
\begin{equation*}
    \zeta_{\Sigma,g}(-2)\geq
 \pi^2B_{l_A}(0) \left(\frac{1}{r_1}+\frac{1}{r_2}\right)^2=\zeta_{\d A_{r_1, r_2,h}}(-2),
\end{equation*}
which proves the first half of the theorem.

Equality holds precisely when all coefficients $a_k^{(\nu)}$ vanish except $a_0^{(1)}$ and $a_0^{(-1)}$, which is equivalent to the function $a$ being constant on each connected component of $\Sigma$.  In other words, 
\[
\phi=\log\frac{r_i}{2\pi} \quad \text{ on }\Sigma_i. 
\]

On the other hand,  if we assume that the metric $g=e^{2\phi}g_0$ is  flat, then the function $\phi$ is necessarily harmonic. Consequently, such $\phi$ is unique and $S$ is isometric to the lateral surface of a conical frustum $A_{r_1, r_2,h}$.
\end{proof}

We now turn to the proof of Theorem \ref{main1}.
\begin{proof}[Proof of Theorem \ref{main1}]
 Note that the lengths $r_i$ of the connected components of the boundary 
are Steklov spectral invariants, and hence are determined by the Steklov spectrum. On the other hand, by Corollary \ref{specdetcon}, all annular Steklov isospectral surfaces  lie in the same conformal class.
Therefore, if $S\in \mathcal{A}$ is a  flat surface that has the same Steklov spectrum as $A_{r_1, r_2, h}$, then necessarily $S\in \mathcal A_{r_1, r_2, h}^{\text{\tiny flat}}$. 
The conclusion follows from the second assertion in Theorem \ref{thm::compactnessdoubly}. 
\end{proof}

\section{The $C^\infty$-compactness of Steklov isospectral surfaces}\label{compactness}
\subsection{A tracial inequality} 

We first recall the following classical convexity lemma, which follows from the spectral theorem and Jensen’s inequality. %For the reader’s convenience, we provide a  proof.
\begin{lemma}\label{jensen}
Let $H$ be a Hilbert space, and let $A$ be a self-adjoint operator on $H$.
Let $f$ be a convex function defined on an interval $I\subset \mathbb R$
containing $\sigma(A)$. Then, for every
$u\in \mathrm{Dom}(A)\cap \mathrm{Dom}(f(A))$ with $\|u\|_H=1$, one has
\[
\langle f(A)u,u\rangle_H
\geq
f(\langle Au,u\rangle_H).
\]
\end{lemma}
\begin{proof}
Let $E(\lambda)$ be the spectral family of $A$. By the spectral theorem,
\[
\langle f(A)u,u\rangle_H
=
\int_{\sigma(A)} f(\lambda)\,d\langle E(\lambda)u,u\rangle_H,
\qquad
\langle Au,u\rangle_H
=
\int_{\sigma(A)} \lambda\,d\langle E(\lambda)u,u\rangle_H.
\]
Since $\|u\|_H=1$, $d\langle E(\lambda)u,u\rangle_H$ is a probability
measure on $\sigma(A)$. Jensen's inequality then implies
\[
\langle f(A)u,u\rangle_H
\geq
f(\langle Au,u\rangle_H),
\]
as claimed.
\end{proof}

We next prove a basic positivity property for a trace difference on the
boundary of the flat cylinder. 
\begin{lemma}\label{positivetrace}
    Let  $\mathcal{C}_l=S^1\times [0,l]$ be a cylinder with boundary $\Sigma=\Sigma_1\cup\Sigma_2$ equipped with the standard product metric $g_0$. Then, for any $a\in C^\infty(\Sigma;\R_+)$ and $s\geq 1$, we have
    \begin{equation*}
    \Tr \left((a^{\frac12}\La a^{\frac12})^{s}-(a^{\frac12}|D|a^{\frac12})^{s}\right)> 0.
    \end{equation*}
\end{lemma}
\begin{proof}
    We first claim that for any nonzero $u\in C^\infty(\Sigma)$ supported on one of the connected components $\Sigma_i$, the following inequality holds:
	\begin{equation*}
		(\Lambda u,u)_{L^2(\Sigma)}>(|D| u,u)_{L^2(\Sigma)}.
	\end{equation*}
Indeed, we write	
\begin{equation*}
		u=\sum_{\nu=\pm1}\sum_{k\in\Z}u_k^{(\nu)}\psi_k^{(\nu)},
\end{equation*}
    where \[u_k^{(\nu)}=\left(u,\psi_k^{(\nu)}\right)_{L^2(\Sigma)}.\]
	Suppose $u=0$ on $\Sigma_1$. Then $u_k^{(1)}=u_k^{(-1)}$ holds for all $k\in \Z$. Therefore 
	\begin{equation*}
		(\Lambda u,u)_{L^2(\Sigma)}=\sum_{k\in \Z}\left(\sigma_k^{(1)}+\sigma_k^{(-1)}\right)\left|u_k^{(1)}\right|^2>\sum_{k\in \Z}2|k|\left|u_k^{(1)}\right|^2=(|D| u,u)_{L^2(\Sigma)}.
	\end{equation*}
It follows that, for any nonzero $u\in C^\infty(\Sigma)$ supported on one boundary component and $a\in C^\infty(\Sigma;\R_+)$, 
    	\begin{equation*}
		(a^{\frac12}\Lambda a^{\frac12} u,u)_{L^2(\Sigma)}>(a^{\frac12}|D|a^{\frac12} u,u)_{L^2(\Sigma)}.
	\end{equation*}
    
Let $\eta_k$ denote the normalized eigenfunctions of the positive elliptic operator $a^{\frac12}|D|a^{\frac12}$ such that each $\eta_k$ is supported on one of the connected components of $\Sigma$. Then by Lemma \ref{jensen}, we have
\begin{align*}
    \left((a^{\frac12}\La a^{\frac12})^{s}\eta_k,\eta_k\right)_{L^2(\Sigma)}&\geq \left(a^{\frac12}\La a^{\frac12}\eta_k,\eta_k\right)^s_{L^2(\Sigma)}\\
    &>\left(a^{\frac12}|D| a^{\frac12}\eta_k,\eta_k\right)^s_{L^2(\Sigma)}\\
    &=\left((a^{\frac12}|D| a^{\frac12})^{s}\eta_k,\eta_k\right)_{L^2(\Sigma)}.
\end{align*}
Summing over $k$ yields
\[
\Tr\left((a^{\frac12}\Lambda a^{\frac12})^s
-(a^{\frac12}|D|a^{\frac12})^s\right)>0.
\]
This completes the proof.    
\end{proof}

\subsection{Proof of Theorem \ref{main::compactness}}

We establish the following key estimate, which controls Sobolev norms of the conformal factor in terms of spectral quantities and will play a central role in the proof of Theorem \ref{main::compactness}.
\begin{proposition}\label{prop::ineq1}
    Let $\mathcal{C}_l=S^1\times [0,l]$ be a cylinder with boundary $\Sigma=\Sigma_1\cup\Sigma_2$. Then, for each $m\in\mathbb Z_+$, there exist
positive constants $d_{m,l}$ and $d'_{m}$ such that, for every positive function
$a\in C^\infty(\Sigma;\mathbb R_+)$,
    \[
    \|a^m\|^2_{H^{m+{\frac12}}(\Sigma)}\leq d_{m,l}\left(\Tr \left((a\Lambda)^2-(aD)^2\right)\right)^{m}+d'_{m} \Tr\left((a\Lambda)^{2m}-(aD)^{2m}\right).
    \]
\end{proposition}
\begin{proof}
Since on $\mathcal{C}_l$ the operators $a\Lambda$ and $a|D|$ have the same full symbol,
 we can write
    \begin{equation*}
   \Tr \left((a\La)^{2m}-(aD)^{2m}\right)=\Tr \left((a\La)^{2m}-(a|D|)^{2m}\right)+\Tr \left((a|D|)^{2m}-(aD)^{2m}\right).
\end{equation*}

By Lemma \ref{positivetrace}, for each $m\in \Z_+$, we have
\begin{equation}
    \Tr \left((a\La)^{2m}-(a|D|)^{2m}\right)=\Tr\left((a^{\frac12}\La a^{\frac12})^{2m}-(a^{\frac12}|D| a^{\frac12})^{2m}\right)> 0.
\end{equation}

On the other hand, each connected component of the boundary $\Sigma$ is  the boundary of a unit disk $\D$, namely, the top or the bottom of the cylinder. Let $g_E$ be the canonical Euclidean metric on that disk. For any smooth functions $f_i$ defined on $\Sigma$ and supported only on one of the connected components $\Sigma_i$, $i=1,2$, we have
\begin{equation}\label{eq::idopDLA}
    (|D|f_i)\big|_{\Sigma_i}=\Lambda_{\d \D,g_E}(f_i|_{\Sigma_i}).
\end{equation}
It follows that
\begin{equation}\label{eq::zeta-2mtwoterms}
    \Tr \left((a|D|)^{2m}-(aD)^{2m}\right)=\sum_{i=1}^2\Tr_{\Sigma_i}\left((a_i\Lambda_{\d \D,g_E})^{2m}-(a_iD_{\d \D,g_E})^{2m}\right),
\end{equation}
where $a_i=a|_{\Sigma_i}$, $i=1,2$. Note that each term on the right-hand side of the equality \eqref{eq::zeta-2mtwoterms} is non-negative. Indeed, by \cite[Theorem 1.1]{JOLLIVET20181712}, there exists a constant $c_m>0$ such that
\begin{equation}\label{eq::tracenorm}
    \Tr_{\Sigma_i}\left((a_i\Lambda_{\d \D,g_E})^{2m}-(a_iD_{\d \D,g_E})^{2m}\right)\geq  c_m\sum_{k=m+1}^\infty k^{2m+1}|(\widehat{a_i^m})_k|^2
\end{equation}
for $i=1,2$ and $m\in \Z_+$, where we use the notation   $\hat{b}_k$ to represent the $k$-th Fourier coefficient in the Fourier expansion of a function $b\in C^\infty (\d \D)$, namely 
\begin{equation*}
    b=\sum_{k=-\infty}^\infty \hat{b}_ke^{ik\theta }.
\end{equation*}
Recall that the norm of the Hilbert space $H^s(\d\D)$ is defined by
\begin{equation*}
    \|b\|^2_{H^s(\d\D)}=\sum_{k=-\infty}^\infty \left(1+|k|^{2s}\right)|\hat{b}_k|^2
\end{equation*}
for all real $s$. Therefore, by \eqref{eq::tracenorm} we have, for $i=1,2$,
\begin{align}
     \|a_i\|^2_{H^{\frac32}(\d\D)}&\leq 4\sum_{k=1}^\infty k^3|(\widehat{a_i})_k|^2+|(\widehat{a_i})_0|^2 \notag\\
     &\leq \frac{4}{c_1}\Tr_{\Sigma_i}\left((a_i\Lambda_{\d \D,g_E})^{2}-(a_iD_{\d \D,g_E})^{2}\right)+4|(\widehat{a_i})_1|^2+|(\widehat{a_i})_0|^2 \label{eq::237}.
\end{align}

On the other hand, from the definition \eqref{eq::defofak},  
\begin{equation*}
\left\{\begin{aligned}
&(\widehat{a_1})_1=    a^{(1)}_1-a^{(-1)}_1, \\
&(\widehat{a_2})_1=     a^{(1)}_1+a^{(-1)}_1,
\end{aligned}
\right.
\quad\text{and}\quad 
\left|(\widehat{a_1})_0\right|^2+\left|(\widehat{a_2})_0\right|^2\leq    4\left( a^{(1)}_0\right)^2.
\end{equation*}
Hence, by \eqref{eq::zetainvarcal} and \eqref{eq::zeta-2}, the three lowest
Fourier modes are controlled by the Steklov spectral invariant
\(\Tr\left((a\Lambda)^2-(aD)^2\right)\). More precisely, define 
\[
r_l:=\max\left\{\frac{8}{B_l(1)},\,\frac{8}{B_l'(1)},\,\frac{4}{B_l(0)}
\right\}.
\]
Then
\begin{align}
\sum_{i=1}^2\left(4\left|(\widehat{a_i})_1\right|^2+\left|(\widehat{a_i})_0\right|^2\right)&=8\left|a^{(1)}_1\right|^2+8\left|a^{(-1)}_1\right|^2+4\left( a^{(1)}_0\right)^2\notag\\
&\leq r_l\Tr \left((a\Lambda)^2-(aD)^2\right).\label{eq::240}
\end{align}

Combining \eqref{eq::zeta-2mtwoterms}, \eqref{eq::237} and \eqref{eq::240}, we obtain that 
\begin{equation*}
    \|a\|^2_{H^{\frac32}(\Sigma)}\leq \left(\frac{4}{c_1}+r_l\right)\Tr \left((a\Lambda)^2-(aD)^2\right).
\end{equation*}
By the embedding theorem $H^{\frac32}(\Sigma)\subset C(\Sigma)$ we have
\begin{equation}\label{eq::aupperbound}
    a\leq C_{\mathrm{emb}}\left(\frac{4}{c_1}+r_l\right) \left(\Tr \left((a\Lambda)^2-(aD)^2\right)\right)^{{\frac12}},
\end{equation}
where $C_{\mathrm{emb}}>0$ is independent of $l$. Set \[r'_l:=C_{\mathrm{emb}}\left(\frac{4}{c_1}+r_l\right).\] It follows that for $i=1,2$ and for any $m \in \mathbb Z_+$,
\begin{align*}
    \|a_i^m\|_{H^{m+{\frac12}}(\d \D)}^2=&\sum_{|k|\leq m}\left(1+|k|^{2m+1}\right)|(\widehat{a_i^m})_k|^2+2\sum_{k= m+1}^\infty\left(1+|k|^{2m+1}\right)|(\widehat{a_i^m})_k|^2\notag\\
    \leq &(2m+1)(1+m^{2m+1})\left(r'_l\right)^{2m}\left(\Tr \left((a\Lambda)^2-(aD)^2\right)\right)^m\\
    &+\frac 4{c_m}\Tr_{\Sigma_i}\left((a_i\Lambda_{\d \D,g_E})^{2m}-(a_iD_{\d \D,g_E})^{2m}\right).
\end{align*}
This estimate shows that for some positive constants $d_{m,l}$ and $d'_{m}$, we have
\begin{align*}
    \|a^m\|^2_{H^{m+{\frac12}}(\Sigma)}&\leq d_{m,l}\left(\Tr \left((a\Lambda)^2-(aD)^2\right)\right)^m+d'_{m}\Tr\left((a|D|)^{2m}-(aD)^{2m}\right)\\
    &< d_{m,l}\left(\Tr \left((a\Lambda)^2-(aD)^2\right)\right)^m+d'_{m}\Tr\left((a\Lambda)^{2m}-(aD)^{2m}\right).\label{eq::bound2.46}
\end{align*}
\end{proof}

The following proposition provides a lower bound for $a$ in terms of the spectral zeta function.
\begin{proposition}\label{prop::ineq2}
Let $\mathcal{C}_l=S^1\times [0,l]$ be a cylinder with boundary
$\Sigma=\Sigma_1\cup\Sigma_2$. Then, for every positive function
$a\in C^\infty(\Sigma;\mathbb R_+)$, one has the pointwise lower bound
\[
a\geq K_l\left(
\int_\Sigma a^{-1}\,d\theta,\,
\operatorname{Tr}\left(a^{\frac12}\Lambda a^{\frac12}
-\left|a^{\frac12}D a^{\frac12}\right|\right),\,
\operatorname{Tr}\left((a\Lambda)^2-(aD)^2\right)
\right),
\]
where
\[
K_l(x,y,z):=
\frac{2\pi}{x}
\exp\left[
-\sqrt{2\pi x}\left(6y+\frac{2}{\sqrt{B_l(0)}}\sqrt{z}\right)^{\frac12}
\right].
\]
\end{proposition}

\begin{proof}
Write $a_i:=a|_{\Sigma_i}$, $i=1,2$. Since the operators
$a\Lambda$ and $a|D|$ have the same full symbol on $\mathcal{C}_l$,
we have
\begin{align*}
&\operatorname{Tr}\left(a^{\frac12}\Lambda a^{\frac12}
-\left|a^{\frac12}D a^{\frac12}\right|\right) \\
&\quad =
\operatorname{Tr}\left(a^{\frac12}\Lambda a^{\frac12}
-a^{\frac12}|D|a^{\frac12}\right)
+
\operatorname{Tr}\left(a^{\frac12}|D|a^{\frac12}
-\left|a^{\frac12}D a^{\frac12}\right|\right).
\end{align*}
The first term on the right-hand side is positive by
Lemma \ref{positivetrace}. For the second term, using
\eqref{eq::idopDLA}, we obtain
\[
\operatorname{Tr}\left(a^{\frac12}|D|a^{\frac12}
-\left|a^{\frac12}D a^{\frac12}\right|\right)
=
\sum_{i=1}^2
\operatorname{Tr}_{\Sigma_i}\left(
a_i^{\frac12}\Lambda_{\partial\mathbb D,g_E}a_i^{\frac12}
-
\left|a_i^{\frac12}D_{\partial\mathbb D,g_E}a_i^{\frac12}\right|
\right).
\]
By \cite[Lemma 2.2 and (2.21)]{Sharafutdinov18}, each term in this
sum is non-negative. In particular,
\[
\operatorname{Tr}\left(a^{\frac12}\Lambda a^{\frac12}
-\left|a^{\frac12}D a^{\frac12}\right|\right)\geq 0.
\]

On the other hand, by \eqref{eq::defofak} and \eqref{eq::zeta-2}, we have
\begin{equation}\label{eq::aicontrol}
\frac{1}{2\pi}\int_{\Sigma_i}a_i\,d\theta
=(\widehat a_i)_0
\leq 2a_0^{(1)}
\leq
\frac{2}{\sqrt{B_l(0)}}
\left(
\operatorname{Tr}\left((a\Lambda)^2-(aD)^2\right)
\right)^{\frac12}.
\end{equation}

We now follow the proof of \cite[Proposition 4]{Edward01011993}, with
a minor modification. Let $g_E$ be the standard Euclidean metric on
$\mathbb D$, and choose $v_i\in C^\infty(\mathbb D;\mathbb R)$ such
that
\[
\widetilde g_i=e^{2v_i}g_E,
\qquad
e^{-v_i}|_{\partial\mathbb D}=a_i .
\]
By \cite{Kogan1979TraceFormulas},
\begin{equation}\label{eq::kogan}
\zeta_{\partial\mathbb D,\widetilde g_i}(-1)
=
\frac{1}{12\pi}
\int_{\Sigma_i}
\left(
\frac{|Da_i|^2}{a_i}-a_i
\right)\,d\theta .
\end{equation}
Since $a_i$ is smooth and positive, there exists
$\theta_0\in[0,2\pi)$ such that
\[
2\pi\,a_i(\theta_0)^{-1}
=
\int_{\Sigma_i}a_i^{-1}\,d\theta
=:r_i .
\]
For any $\theta_1\in[0,2\pi)$,  the Cauchy--Schwarz inequality, together with \eqref{eq::aicontrol} and \eqref{eq::kogan}, gives
\begin{align*}
\left|\log a_i(\theta_1)-\log a_i(\theta_0)\right|
&=\left|\int_{\theta_0}^{\theta_1}\frac{Da_i}{a_i}\,d\theta\right| \\
&\leq\left(\int_{\theta_0}^{\theta_1}a_i^{-1}\,d\theta
\right)^{\frac12}
\left(\int_{\theta_0}^{\theta_1}\frac{|Da_i|^2}{a_i}\,d\theta
\right)^{\frac12} \\
&\leq\sqrt{2\pi r_i}
\left(6\zeta_{\partial\mathbb D,\widetilde g_i}(-1)+\frac{2}{\sqrt{B_l(0)}}
\left(\operatorname{Tr}\left((a\Lambda)^2-(aD)^2\right)\right)^{\frac12}\right)^{\frac12}.
\end{align*}
Hence
\[
\log a_i(\theta_1)\geq\log\frac{2\pi}{r_i}-\sqrt{2\pi r_i}
\left(6\zeta_{\partial\mathbb D,\widetilde g_i}(-1)
+\frac{2}{\sqrt{B_l(0)}}
\left(\operatorname{Tr}\left((a\Lambda)^2-(aD)^2\right)\right)^{\frac12}\right)^{\frac12}.
\]
Therefore
\begin{align*}
a_i(\theta_1)
&\geq\frac{2\pi}{r_i}
\exp\left[-\sqrt{2\pi r_i}\left(6\zeta_{\partial\mathbb D,\widetilde g_i}(-1)+\frac{2}{\sqrt{B_l(0)}}\left(\operatorname{Tr}\left((a\Lambda)^2-(aD)^2\right)\right)^{\frac12}\right)^{\frac12}\right] \\
&=K_l\left(r_i,\,\zeta_{\partial\mathbb D,\widetilde g_i}(-1),\,
\operatorname{Tr}\left((a\Lambda)^2-(aD)^2\right)\right).
\end{align*}

It remains to replace the componentwise quantities by global ones.
By \eqref{generalzeta-1} and \eqref{generalzeta-2m},
\begin{align*}
\zeta_{\partial\mathbb D,\widetilde g_i}(-1)
&=\operatorname{Tr}_{\Sigma_i}\left(a_i^{\frac12}\Lambda_{\partial\mathbb D,g_E}a_i^{\frac12}-\left|a_i^{\frac12}D_{\partial\mathbb D,g_E}a_i^{\frac12}\right|\right)-\frac{\pi}{3}\left(\int_{\Sigma_i}a_i^{-1}\,d\theta\right)^{-1} \\
&\leq\operatorname{Tr}\left(
a^{\frac12}\Lambda a^{\frac12}-\left|a^{\frac12}D a^{\frac12}\right|\right).
\end{align*}
Also,
\[
r_i=\int_{\Sigma_i}a_i^{-1}\,d\theta\leq\int_\Sigma a^{-1}\,d\theta .
\]
Since $K_l(x,y,z)$ is positive and decreasing in each variable
$x,y,z$, it follows that for $i=1,2$,
\[
a_i\geq K_l\left(
\int_\Sigma a^{-1}\,d\theta,\,
\operatorname{Tr}\left(a^{\frac12}\Lambda a^{\frac12}
-\left|a^{\frac12}D a^{\frac12}\right|
\right),\,\operatorname{Tr}\left((a\Lambda)^2-(aD)^2\right)\right).
\]
This completes the proof.
\end{proof}
We are now ready to prove the main compactness theorem by combining the preceding propositions.
\begin{proof}[Proof of Theorem \ref{main::compactness}]
Let $\mathcal{F}\subset \mathcal{A}$ denote a family of flat annular surfaces with identical Steklov spectrum. For any annular surface $S\in\mathcal{F}$, by Proposition \ref{prop::zetaconformalinv} and the discussion in Section \ref{uniformazation}, $S$ is isometric to some $(\mathcal{C}_l,g)$, where $g=e^{2\phi}g_0$, $\phi\in C^\infty (\mathcal{C}_l,\R)$ and $l$ can be determined by the Steklov spectrum of $S$. Let $a=e^{-\phi}|_{\Sigma}$. 
 By \eqref{eq::aupperbound}, once the Steklov spectrum of $S$ is given, we have a uniform upper bound for $a$,
 \[
 a\leq r'_l \left(\zeta_{\Sigma,g}(-2)\right)^\frac12.
 \]
Additionally, by \eqref{eq::zeta2-1} and Proposition \ref{prop::ineq2},  $a$ is bounded below uniformly by  
\[
a\geq K_l\left(L_g(M),\zeta_{\Sigma,g}(-1)+\frac{\pi}{3}\sum_{i=1}^n \frac1{L_g(M_i)},\zeta_{\Sigma,g}(-2)\right)>0.
\]

On the other hand, by \eqref{eq::zetainvarcal} and Proposition
\ref{prop::ineq1}, for each $m\in\Z_+$ we have 
\begin{equation}\label{eq::conclusion2m}
    \|a^m\|^2_{H^{m+{\frac12}}(\Sigma)}\leq d_{m,l}\left(\zeta_{\Sigma,g}(-2)\right)^{m}+d'_{m}\zeta_{\Sigma,g}(-2m).
\end{equation}
For smooth functions bounded uniformly from above and below away from zero, the equivalence of Sobolev norms $1+\|\cdot\|_{H^s(\Sigma)}$ and $1+\|\log(\cdot)\|_{H^s(\Sigma)}$ leads to the inequality
\begin{equation}\label{eq::2.51}
    \|\phi\|_{H^{m}(\Sigma)}=\|\log a\|_{H^{m}(\Sigma)}  =\frac{1}{m}\|\log a^m\|_{H^{m}(\Sigma)}\leq {C_m}\left(1+\|a^m\|_{H^{m+\frac12}(\Sigma)}\right).
\end{equation}
Since $g$ is a flat metric, the function $\phi$ is harmonic on $\mathcal{C}_l$. Hence, by the standard elliptic estimate, there exists a constant $d''_{m,l}>0$ such that
\begin{equation}\label{eq::2.52}
    \|\phi\|_{H^m(\mathcal{C}_l)}\leq d''_{m,l}\|\phi\|_{H^{m}(\Sigma)}.
\end{equation}
Combining \eqref{eq::conclusion2m}, \eqref{eq::2.51} and \eqref{eq::2.52}, for any metric $g=e^{2\phi}g_0$ with the given Steklov spectrum,
 the conformal factor $\phi$ is uniformly bounded in $H^m(\mathcal{C}_l)$ for each $m\in \Z_+$. By Rellich's compactness theorem, the Sobolev embedding theorem, and a diagonal argument, the corresponding
Steklov isospectral metrics form a precompact set in the $C^\infty$ topology.
Consequently, $\mathcal F$ is compact in the $C^\infty$ topology.
\end{proof}

%\renewcommand{\refname}{References}
%\bibliographystyle{amsalpha}
%\bibliography{ref}
\end{document}